\documentclass[a4paper]{article}
\usepackage[utf8]{inputenc}

\usepackage[margin=2.5cm]{geometry}

\usepackage{amsmath, amssymb,amsthm, rotating, mathtools}
\usepackage{xcolor}
\usepackage[hidelinks]{hyperref}

 \usepackage{cleveref}
 \newtheorem{thm}{Theorem}[section]
 \newtheorem{lm}[thm]{Lemma}
 \newtheorem{res}[thm]{Result}
 
 \newtheorem{prop}[thm]{Proposition}

 \theoremstyle{definition}
 
 \newtheorem{rmk}[thm]{Remark}
 \newtheorem{df}[thm]{Definition}

 \newcommand{\eps}{\varepsilon}

 \newcommand{\RR}{\mathbb R}
 \newcommand{\FF}{\mathbb F}

 \newcommand{\vspan}[1]{\left \langle #1 \right \rangle}
 \newcommand{\vspann}[2]{\left \langle #1 \, \| \, #2 \right \rangle}
 \newcommand{\set}[1]{ \left \{ #1 \right \} }
 \newcommand{\sett}[2]{ \left\{ #1 \, \, || \, \, #2 \right \} }

 \newcommand{\pg}{\textnormal{PG}}
 \newcommand{\one}{\mathbf 1}
 \newcommand{\zero}{\mathbf 0}
 
 \newcommand{\ceil}[1]{\left \lceil #1 \right \rceil}

 \newcommand{\pp}{\mathcal P}
 \newcommand{\ma}{\mathcal A}
 \newcommand{\ml}{\mathcal L}
 \newcommand{\mq}{\mathcal Q}
 \newcommand{\mh}{\mathcal H}
 
 \newcommand{\ms}{\mathcal S}
 \newcommand{\mo}{\mathcal O}
 
 \DeclareMathOperator{\GL}{GL}
 \DeclareMathOperator{\PGL}{PGL}
 \DeclareMathOperator{\PSL}{PSL}
 \DeclareMathOperator{\PGO}{PGO}
 \DeclareMathOperator{\SL}{SL}
 \DeclareMathOperator{\ColSp}{ColSp}
 \DeclareMathOperator{\rk}{rk}
 \DeclareMathOperator{\SRG}{SRG}

 \DeclareMathOperator{\Tr}{Tr} 
 \DeclareMathOperator{\id}{id}
 \DeclareMathOperator{\dds}{dds}

\newcommand{\ns}{\overline S}
\newcommand{\comments}[1]{}

\title{Association schemes and orthogonality graphs on anisotropic points of polar spaces}
\author{Sam Adriaensen\thanks{Department of Mathematics and Data Science, Vrije Universiteit Brussel, Pleinlaan 2, 1050 Elsene, Belgium. ORCID: 0000-0002-0925-9305. \href{mailto:sam.adriaensen@vub.be}{\url{sam.adriaensen@vub.be}}} \\ \textit{Vrije Universiteit Brussel} \and Maarten De Boeck\thanks{Department of Mathematical Sciences, University of Memphis, Dunn Hall, Norriswood Ave, Memphis, TN 38111, USA. ORCID: 0000-0001-8399-9064. \href{mailto:mdeboeck@memphis.edu}{\url{mdeboeck@memphis.edu}}\\Department of Mathematics: Algebra and Geometry, Ghent University, Gent, Flanders, Belgium} \\ \textit{University of Memphis}}
\date{}

\begin{document}

\maketitle

\begin{center}
 \emph{Dedicated to the memory of Kai-Uwe Schmidt.}
\end{center}

\begin{abstract}
 In this paper, we study association schemes on the anisotropic points of classical polar spaces.
 Our main result concerns non-degenerate elliptic and hyperbolic quadrics in $\pg(n,q)$ with $q$ odd.
 We define relations on the anisotropic points of such a quadric that depend on the type of line spanned by the points and whether or not they are of the same ``quadratic type''.
 This yields an imprimitive $5$-class association scheme.
 We calculate the matrices of eigenvalues and dual eigenvalues of this scheme.

 We also use this result, together with similar results from the literature concerning other classical polar spaces, to exactly calculate the spectrum of orthogonality graphs on the anisotropic points of non-degenerate quadrics in odd characteristic and of non-degenerate Hermitian varieties.
 As a byproduct, we obtain a 3-class association scheme on the anisotropic points of non-degenerate Hermitian varieties, where the relation containing two points depends on the type of line spanned by these points, and whether or not they are orthogonal.
\end{abstract}



\section{Introduction}

Distance-regular graphs are the combinatorial generalisation of distance-transitive graphs, and as such exhibit a great deal of regularity.
Some of the most famous families of distance-regular graphs arise from finite geometries, since geometries typically satisfy strong regularity conditions, and often have a lot of symmetries.
These families include the Grassmann graphs, dual polar graphs, and collinearity graphs of generalised polygons.
More information can be found in \cite{bcn}.
All of the above examples are graphs whose vertices are subspaces contained in some geometry.
However, there are also interesting graphs whose vertices are subspaces outside of the geometry.
Graphs of the latter type will be the focus of this paper.

Distance-regular graphs form a subclass of association schemes, which are the combinatorial generalisation of generously transitive group actions.
We are interested in association schemes defined on the anisotropic points of polar spaces embedded in finite projective spaces.
Such schemes have been investigated in \cite[\S 12]{bcn} and in \cite{chakravarti71, wei83, brouwervanlint, bannaihaosong, bannaisonghaowei, vanhove20}.
Most of these association schemes are described as follows.
Take a non-degenerate quadric or Hermitian variety $\mq$ in $\pg(n,q)$ and let $\pp$ denote the set of anisotropic points to $\mq$, excluding the nucleus in case $\mq$ is a parabolic quadric and $q$ is even.
Define relations on the points of $\pp$ where the relation containing a pair of distinct anisotropic points $(X,Y)$ depends on $|\vspan{X,Y} \cap \mq|$.
These relations often constitute an association scheme.
Most of the known results can be found in \cite[\S 3.1]{brouwervanmaldeghem} and we also present an overview in \Cref{Sec:Known}.
It is striking that the association schemes have been investigated, and their matrices of (dual) eigenvalues have been determined, for all non-degenerate quadrics and Hermitian varieties $\mq$ in $\pg(n,q)$, except for the case where $q$ is odd and $\mq$ is an elliptic or hyperbolic quadric.
Up to our knowledge, this has only been addressed in \cite[\S 12.2]{bcn} in the 3-dimensional case, see \cite[p.\ 335]{FisherPenttilaPraegerRoyle} for the intersection numbers in the elliptic case and \cite[\S 3]{adriaensen2022stability} for the matrix of eigenvalues in the hyperbolic case, and in \cite[\S 3.1.3]{brouwervanmaldeghem} for the case $q=3$.
In this paper, we complete the picture.
We prove that if $\mq$ is a non-degenerate elliptic or hyperbolic quadric in $\pg(n,q)$ with $q$ odd, then the scheme on the anisotropic points of $\mq$, where the relation containing a pair of distinct anisotropic points $(X,Y)$ depends on
\begin{itemize}
 \item how $\vspan{X,Y}$ intersects the quadric $\mq$,
 \item whether or not $X$ and $Y$ are of the same ``quadratic type'' (see \Cref{Df:QuadraticType}),
\end{itemize}
constitutes a 5-class association scheme (or a 4-class one in case $q=3$).
The (dual) eigenvalues of this scheme are presented in \Cref{Tab:Matrices}.
Note that the relations respect the equivalence relation defined by the quadratic type of the anisotropic points.
Hence, the association scheme is imprimitive.
We can restrict ourselves to the subscheme where we take anisotropic points of one quadratic type.
This yields a primitive 3-class association scheme, except when $q=3$, in which case we obtain the 2-class association scheme described in \cite[\S 3.1.3]{brouwervanmaldeghem}.

We use the eigenvalues we computed, together with the previously computed eigenvalues of similar association schemes, to compute the eigenvalues of the orthogonality graph on the anisotropic points of a non-degenerate quadric, for $q$ odd, or Hermitian variety embedded in $\pg(n,q)$ in \Cref{Sec:Ortho}.
As a byproduct, we obtain a 3-class association scheme on the anisotropic points of a Hermitian variety by splitting the complement of the $NU(n+1,q)$ graph (see \cite[\S 3.1.6]{bcn}) into two parts, see \Cref{Prop:FissionHerm}.

\paragraph{Overview.} \Cref{Sec:Prel} contains detailed preliminaries on finite classical polar spaces and association schemes.
\Cref{Sec:Known} contains an overview of the known association schemes on anisotropic points of the finite classical polar spaces.
In \Cref{Sec:NewScheme} we consider the scheme on the anisotropic points of the non-degenerate elliptic and hyperbolic quadrics embedded in $\pg(n,q)$ with $q$ odd, determine its eigenvalues, and give some combinatorial descriptions of the eigenspaces.
\Cref{Sec:Ortho} contains the eigenvalues of the aforementioned orthogonality graphs on anisotropic points of non-degenerate Hermitian varieties and, in case $q$ is odd, quadrics embedded in $\pg(n,q)$.

\section{Preliminaries}
 \label{Sec:Prel}

Throughout this article, $q$ will denote a prime power, and $\FF_q$ the finite field of order $q$.
If $q$ is odd, we denote the set of non-zero squares of $\FF_q$ as $S_q$, and the set of non-squares of $\FF_q$ as $\ns_q$.
For two sets $A,B \subseteq \FF_q$, we write $A \cdot B = \sett{ab}{a \in A, \, b \in B}$. %
It is generally known that $S_q \cdot \ns_q = \ns_q \cdot S_q = \ns_q$ and $S_q \cdot S_q = \ns_q \cdot \ns_q = S_q$.

The $n$-dimensional projective space over $\FF_q$ will be denoted as $\pg(n,q)$.
Its points correspond to the vector lines of $\FF^{n+1}_q$.
Recall that the number of points in $\pg(n,q)$ equals
\[
 \theta_n(q) \coloneqq \frac{q^{n+1}-1}{q-1} = q^n + q^{n-1} + \ldots + q + 1.
\]

\subsection{Classical polar spaces}

Polar spaces are a class of incidence geometries, which can be defined in several ways.
The definition due to Buekenhout and Shult \cite{buekenhoutshult} is as follows.
A \emph{point-line geometry} is a tuple $(\pp,\ml)$ of non-empty sets, where we call the elements of $\pp$ points, the elements of $\ml$ lines, and every line is a subset of the points.
We call two points \emph{collinear} if there is a line containing both of them.
A \emph{(singular) subspace} $\pi$ is a set of pairwise collinear points that fully contains any line intersecting it in at least two points.
Its \emph{rank} is the largest integer $r$ such that there exist subspaces $\emptyset \subsetneq \pi_1 \subsetneq \ldots \subsetneq \pi_r \subsetneq \pi$.
\begin{df}
 A \emph{polar space} of rank $r$ is a point-line geometry $(\pp,\ml)$ such that
 \begin{itemize}
  \item given a line $\ell$ and a point $P \notin \ell$, $P$ is collinear with either a unique point of $\ell$ or all points of $\ell$,
  \item the maximum rank of its subspaces is $r-1$,
  \item no point is collinear with all others,
  \item every line contains at least three points.
 \end{itemize}
 The subspaces of rank $r-1$ are called \emph{generators}.
\end{df}

We will now discuss the construction of the finite classical polar spaces.
These are polar spaces embedded in projective spaces over finite fields.
For an in-depth treaty of the subject, we refer the reader to \cite[\S 2]{brouwervanmaldeghem}, \cite[\S\S 1, 2, 5.1]{hirschfeldthas}, \cite[\S 7]{Shult11}, or \cite{BuekenhoutCameron}. 
The subspaces of finite classical polar spaces are projective subspaces and their rank corresponds to their projective dimension.
Most of the classical polar spaces arise from a polarity.

\begin{df}
 A \emph{polarity} of a projective space is an involution $\perp$ acting on the subspaces of the projective space that reverses inclusion, i.e.\ two subspaces $\pi$ and $\rho$ satisfy $\pi \subseteq \rho$ if and only if $\rho^\perp \subseteq \pi^\perp$.
 A subspace $\pi$ is called \emph{totally isotropic} with respect to the polarity $\perp$ if $\pi \subseteq \pi^\perp$.
\end{df}

Let $\pp$ and $\ml$ denote the sets of totally isotropic points and lines of a polarity, respectively.
If $\ml$ is non-empty and $\pp$ is not a hyperplane, then $(\pp,\ml)$ is a polar space.
Throughout the paper, if $\mq$ is a polar space embedded in a projective space $\pg(n,q)$, we call the points of $\mq$ the \emph{isotropic} points, and the points of $\pg(n,q)$ not lying in $\mq$ \emph{anisotropic}.

Let $\sigma$ be an involutary field automorphism of $\FF_q$.
Then $\sigma$ is either the identity, or $q$ is a square and $\sigma: \alpha \to \alpha^{\sqrt q}$.
A map $B: \FF_q^{n+1} \times \FF_q^{n+1} \to \FF_q$ is called \emph{sesquilinear} if it is additive in both components and 
$B(\alpha x, \beta y) = \alpha \beta^\sigma B(x,y)$ for all scalars $\alpha, \beta \in \FF_q$ and all vectors $x,y \in \FF_q^{n+1}$.
Moreover, $B$ is called \emph{reflexive} if there exists a scalar $\gamma$ such that $B(x,y) = \gamma B(y,x)^\sigma$ for all $x,y \in \FF_q^{n+1}$.
We call $B$ a(n)
\begin{itemize}
 \item \emph{symmetric (bilinear) form} if $\sigma = \id$ and $\gamma = 1$,
 \item \emph{alternating (bilinear) form} if $\sigma = \id$, $\gamma=-1$, and $B(x,x) = 0$ for all $x \in \FF_q^{n+1}$,
 \item \emph{Hermitian form} if $\sigma \neq \id$ and $\gamma=1$.
\end{itemize}
In addition, $B$ is \emph{non-degenerate} if for every non-zero vector $x \in \FF_q^{n+1}$ there exists a vector $y \in \FF_q^{n+1}$ such that $B(x,y) \neq 0$.
A non-degenerate reflexive form $B$ gives rise to the polarity $\perp$ on the subspaces of $\FF_q^{n+1}$ (which are the same as the subspaces of $\pg(n,q)$) defined by
\[
 W^\perp = \sett{ x \in \FF_q^{n+1} }{ \left(\forall y \in W \right) (B(x,y) = 0 ) }.
\]

In case the characteristic of the underlying field is odd, all polar spaces arise from a polarity in this way.
In case the characteristic is even, this is no longer true.
However, there is still a unified way to describe the different types of classical polar spaces in both odd and even characteristic.


\subsubsection{Quadrics}

We first discuss the polar spaces related to symmetric forms.
To incorporate fields of even characteristic, we need to take a slightly different approach.

Let $\kappa$ be a quadratic form on the vector space $\FF_q^{n+1}$.
The set $\mq$ of points $X$ in $\pg(n,q)$ whose coordinate vectors $x$ satisfy $\kappa(x) = 0$ is called a \emph{quadric}.
We call $\kappa$ \emph{degenerate} if it can be linearly transformed into a quadratic form depending on less than $n+1$ variables, and \emph{non-degenerate} otherwise.
The quadric $\mq$ is called (non-)degenerate accordingly.
If $\mq$ is non-degenerate, and fully contains some lines of $\pg(n,q)$, then the points of $\mq$ together with these lines constitute a polar space.

Associated to $\kappa$ is the bilinear form $B(x,y) = \kappa(x+y) - \kappa(x) - \kappa(y)$.
Despite the non-degeneracy of $\kappa$, this bilinear form can still be degenerate.
If $q$ is odd, $B$ is never degenerate, and $\mq$ consists of the isotropic points of the polarity arising from $B$.
If $q$ is even and $n$ is odd, $B$ is non-degenerate, but every point is isotropic with respect to the associated polarity.
If $q$ and $n$ are both even, then $B$ is degenerate.

If $q$ is odd, then for any projective point $X$ one of the following three options holds:
\begin{enumerate}
 \item every vector representative $x$ of $X$ satisfies $\kappa(x) \in S_q$, which we denote by $\kappa(X) = S_q$,
 \item every vector representative $x$ of $X$ satisfies $\kappa(x) \in \ns_q$, which we denote by $\kappa(X) = \ns_q$,
 \item every vector representative $x$ of $X$ satisfies $\kappa(x) = 0$, which we denote by $\kappa(X) = 0$, or equivalently $X \in \mq$.
\end{enumerate}

\begin{df}
 \label{Df:QuadraticType}
    Given a quadratic form $\kappa$ on $\FF_q^{n+1}$ and a point $X$ in $\pg(n,q)$, we call $\kappa(X)$ the \emph{(quadratic) type} of $X$.
\end{df}

There are three classes of non-degenerate quadrics.

\begin{itemize}
 \item If $n$ is even, $\pg(n,q)$ contains up to isomorphism a unique non-degenerate quadric, defined by the quadratic form
\[
 \kappa(x) = x_0 x_1 + \ldots + x_{n-2} x_{n-1} + x_n^2.
\]
This quadric is called \emph{parabolic}, and denoted $\mq(n,q)$.
It constitutes a polar space of rank $\frac n2$.
If $q$ is even, then the bilinear form $B(x,y) = \kappa(x+y) - \kappa(x) - \kappa(y)$ is degenerate.
Indeed, if $x=(0,\dots,0,1)$, then $B(x,y) = 0$ for all $y \in \FF_q^{n+1}$.
The projective point $X$ corresponding to $x=(0,\dots,0,1)$ is called the \emph{nucleus} of $\mq(n,q)$.

 \item If $n$ is odd, $\pg(n,q)$ contains up to isomorphism two non-degenerate quadrics.
They are both defined by a quadratic form
\[
 \kappa(x) = x_0 x_1 + \ldots + x_{n-3} x_{n-2} + f(x_{n-1},x_n),
\]
for some non-degenerate quadratic form $f$.
If $f$ is irreducible, the quadric is called \emph{elliptic}, and denoted as $\mq^-(n,q)$.
It is of rank $\frac {n-1}2$.
If $f$ is reducible, the quadric is called \emph{hyperbolic}, and denoted as $\mq^+(n,q)$.
It is of rank $\frac {n+1}2$.
\end{itemize}

We will often discuss the different quadrics simultaneously.
We will denote them as $\mq^\eps(n,q)$, with $\eps \in \{0, \pm 1\}$; the notations $\pm$ and $\pm 1$ will be used interchangeably.
The quadric is hyperbolic, parabolic, or elliptic when $\eps$ equals $1,0,-1$, respectively, where we always assume that $\eps \equiv n \pmod 2$.

We will give the formula for the number of points on $\mq^\eps(n,q)$.

\begin{res}[{\cite[Theorem 1.41]{hirschfeldthas}}]
 \label{ResQuadricSize}
 The number of points on $\mq^\eps(n,q)$ equals
 \[
  \frac 1 {q-1} \left( q^\frac{n+\eps}2 - 1 \right) \left( q^\frac{n-\eps}2 + 1 \right) = \theta_{n-1}(q) + \eps q^\frac{n-1}2.
 \]
 Therefore, the number of anisotropic points equals
 \[
  q^\frac{n-1}2 \left( q^\frac{n+1}2 - \eps \right).
 \]
\end{res}

Next we describe the intersection of quadrics with subspaces.

\begin{df}
 Let $\pi$ and $\rho$ be two disjoint subspaces in $\pg(n,q)$ and let $S$ be a subset of the points of $\pi$.
 The \emph{cone} $\rho S$ with \emph{vertex} $\rho$ and \emph{base} $S$ is the set of points
 \[
  \bigcup_{P \in S} \vspan{P,\rho}.
 \]
 By convention, $\rho S = S$ if $\rho = \emptyset$ and $\rho S = \rho$ if $S = \emptyset$.
\end{df}

A cone $\rho \mq$ with as base a quadric $\mq$ is again a quadric.
This quadric is degenerate unless $\rho = \emptyset$ and $\mq$ is non-degenerate.
If $\mq$ is a quadric in $\pg(n,q)$, then for every subspace $\pi$, $\mq \cap \pi$ is a quadric in $\pi$.
We will use the notation $\Pi_m \mq^\eps(n,q)$ to denote a quadric whose vertex is $m$-dimensional and whose base is a $\mq^\eps(n,q)$.

\begin{res}[{\cite[\S 1.7]{hirschfeldthas}}]
 Consider the non-degenerate quadric $\mq^\eps(n,q)$ in $\pg(n,q)$.
 Suppose that $n$ and $q$ are not both even, so that there exists a polarity $\perp$ associated to $\mq^\eps(n,q)$.
 Let $\pi$ be a subspace of $\pg(n,q)$.
 \begin{enumerate}
  \item If $\pi \subset \mq^\eps(n,q)$, then $\pi^\perp \cap \mq^\eps(n,q) \cong \pi \mq^\eps(n-2\dim(\pi)-2,q)$.
  \item If $\eps = \pm 1$ and $\pi \cap \mq^\eps(n,q) \cong \Pi_{m_1} \mq^\delta(m_2,q)$, then $\pi^\perp \cap \mq^\eps(n,q) \cong \Pi_{m_1} \mq^{\delta \eps}(n-2m_1 - m_2 - 3,q)$.
  \item If $\eps = 0$, then $\pi \cap \mq(n,q) \cong \Pi_{m_1} \mq(m_2,q)$ if and only if $\pi^\perp \cap \mq(n,q) \cong \Pi_{m_1} \mq^{\pm}(n-2m_1-m_2-3,q)$.
  \item If $q$ is odd, then $\pi$ intersects $\mq^\eps(n,q)$ in a cone with vertex $\pi \cap \pi^\perp$ and base a non-degenerate quadric.
 \end{enumerate}
\end{res}

A quadric $\mq$ intersects every line that is not totally isotropic in at most 2 points.
Lines that contain 0, 1, or 2 points of $\mq$ are called \emph{passant}, \emph{tangent}, or \emph{secant}, respectively.

\subsubsection{Hermitian varieties}

Consider the polarity $\perp$ associated to the non-degenerate Hermitian form
\[
 B(x,y) = x_0 y_0^q + \ldots + x_n y_n^q
\]
defined on $\FF_{q^2}^{n+1}$.
The set of isotropic points of this polarity is called the \emph{Hermitian variety} and denoted as $\mh(n,q^2)$.
This gives us an embedded polar space of rank $\ceil {\frac n2}$.
The related form $\kappa(x) = B(x,x)$ is called \emph{pseudo-quadratic}.

\begin{res}[{\cite[Theorem 2.8]{hirschfeldthas}}]\label{Res:HermitianAnisotropicPoints}
 The number of points on $\mh(n,q^2)$ equals
 \[
  \frac 1 {q^2-1} \left( q^{n+1} + (-1)^n \right) \left( q^n - (-1)^n \right) =  q^n \frac{q^n-(-1)^n}{q+1} + \theta_{n-1}(q^2).
 \]
 Therefore the number of anisotropic points equals
 \[
  q^n \frac{q^{n+1} + (-1)^n}{q+1}.
 \]
\end{res}

\subsubsection{Symplectic polar spaces}

Symplectic polar spaces arise from polarities associated to non-degenerate alternating forms.
With respect to these forms, every point is isotropic.
Since we are interested in the geometry of anisotropic points, we will not discuss the symplectic polar spaces.

\subsection{Circle geometries}

Circle geometries (sometimes also called \emph{Benz planes}) are a class of point-line geometries.
The lines of a circle geometry are often called \emph{circles}.
For a survey on circle geometries, see e.g.\ Hartmann \cite{hartmann} or Delandtsheer \cite[\S 5]{delandtsheer}.

We first introduce the concept of a \emph{parallel relation} for a point-line geometry.
This is an equivalence relation on the point set; its equivalence classes are called \emph{parallel classes}.
We call two points \emph{parallel} if they are in the same parallel class of some parallel relation.

\begin{df}
 A \emph{circle geometry} is a point-line geometry $(\pp,\ml)$ with at most 2 parallel relations, such that the following properties hold:
\begin{enumerate}
 \item Given 3 pairwise non-parallel points, there is a unique circle containing these three points.
 \item Given a circle $c$, a point $P \in c$, and a point $Q \notin c$, not parallel with $P$, there is a unique circle through $P$ and $Q$, which is \emph{tangent} to $c$, i.e.\ it intersects $c$ only in $P$.
 \item Every parallel class contains a unique point of each circle.
 \item Two parallel classes from different parallel relations intersect in a unique point.
 \item Each circle contains at least 3 points.
 \item There exists a circle and a point not on this circle.
\end{enumerate}
A circle geometry with 0, 1, or 2 parallel relations is called a \emph{Möbius plane} (or \emph{inversive plane}), \emph{Laguerre plane}, or a \emph{Minkowski plane}, respectively.
\end{df}

\begin{rmk}
 In a Möbius plane, no two points are parallel.
 In this case, (1) should be read as ``Given 3 points, there is a unique circle containing these three points.''
 Likewise, for Möbius planes, in (2), $Q$ is allowed to be any point not on $c$, and (3) is a vacuous statement.
 Property (4) is only relevant for Minkowski planes.
\end{rmk}

It is a classic result that every circle in a finite circle geometry has the same number of points.
If this number is $q+1$, then $q$ is called the \emph{order} of the circle geometry.
In particular, a Möbius plane of order $q$ is a $3-(q^2+1,q+1,1)$ design.
The classical construction of a finite circle geometry of order $q$ is taking a quadric $\mq$ in $\pg(3,q)$ which is either $\mq^{\pm}(3,q)$ or $\Pi_0 \mq(2,q)$, taking as $\pp$ the points of this quadric, excluding the vertex in case of a degenerate parabolic quadric, and as $\ml$ the non-degenerate plane sections of the quadric.
Circle geometries arising in this way are called \emph{miquelian} since they are characterised by an extra regularity condition known as the theorem of Miquel, see e.g.\ \cite[\S 5.8, 5.10]{delandtsheer}.

The three types of miquelian circle geometries all have interesting alternative representations.
\begin{enumerate}
 \item Let $\pp$ be the set of points of $\pg(1,q^2)$.
 The projective line $\pg(1,q)$ can be considered to be a subset of $\pg(1,q^2)$ in a canonical way, namely the set of the points of $\pg(1,q^2)$ having a coordinate vector in $\FF_q^2$.
 The images of $\pg(1,q)$ under the action of $\PGL(2,q^2)$ are called the \emph{Baer sublines} of $\pg(1,q^2)$.
 Let $\ml$ be the set of Baer sublines of $\pg(1,q^2)$.
 Then the point-line geometry $(\pp,\ml)$ is isomorphic to the miquelian Möbius plane of order $q$.
 \item For any function $f:A \to B$ define its \emph{graph} to be the set $\sett{(x,f(x))}{ x \in A} \subset A \times B$.
 Let $\pp$ be the set $(\FF_q \cup \{\infty\}) \times \FF_q$, and for every polynomial $f(X) = a X^2 + b X + c$ of degree at most 2 define $f(\infty) = a$.
 Let $\ml$ be the set of graphs of the polynomials of $\FF_q[X]$ of degree at most 2, seen as functions $\FF_q \cup \{\infty\} \to \FF_q$.
 Then $(\pp,\ml)$ is isomorphic to the miquelian Laguerre plane of order $q$.
 \item Let $\pp$ be $\pg(1,q) \times \pg(1,q)$ and let $\ml$ be the set of graphs of the elements of $\PGL(2,q)$.
 Then $(\pp,\ml)$ is isomorphic to the miquelian Minkowski plane of order $q$.
\end{enumerate}

\subsection{Association schemes}
 \label{SubSec:Assoc}

In the literature, there is some variation in the definition of an association scheme.
In this paper we will use the most restrictive definition; other authors might call this a symmetric association scheme.
We will review the basic properties of association schemes.
These can be found for instance in 
\cite[\S\S 2.1-2.2]{bcn} or \cite[\S 3]{godsilmeagher}.

Association schemes have a combinatorial and an algebraic definition, which are equivalent.
We start with the combinatorial definition.
First, let us introduce some convenient notation.

\begin{df}
 Suppose that $\pp$ is a set, $X \in \pp$, and $R \subseteq \pp \times \pp$.
 Then we use the following notation.
 \[
  R(X) = \sett{Y \in \pp}{(X,Y) \in R}.
 \]
\end{df}

\begin{df}
 Consider a set $\pp$ and a partition $R_0, \ldots, R_d$ of $\pp \times \pp$.
 We call this partition a \emph{$d$-class association scheme} if it satisfies the following properties:
 \begin{itemize}
  \item $R_0 = \sett{(X,X)}{X \in \pp}$,
  \item every relation $R_i$ is symmetric in the sense that $R_i = \sett{(Y,X)}{(X,Y) \in R_i}$,
  \item for all $i, j \in \{0,\ldots,d\}$, and for every $X,Y \in \pp$, $|R_i(X) \cap R_j(Y)|$ only depends on which relation $R_k$ contains $(X,Y)$.
  We denote $|R_i(X) \cap R_j(Y)|$ by $p^k_{i,j}$ if $(X,Y) \in R_k$, and call the integers $p^k_{i,j}$ the \emph{intersection numbers}.
 \end{itemize}
\end{df}

The relations $R_i$ can be equivalently expressed by their \emph{adjacency matrices}
\[
 A_i: \pp \times \pp \to \RR: (X,Y) \mapsto \begin{cases}
  1 & \text{if } (X,Y) \in R_i, \\
  0 & \text{otherwise.}
 \end{cases}
\]
By $I$ and $J$ we denote the identity matrix and the all-one matrix, respectively.
The dimensions should be clear from context.
Then the algebraic definition of an association scheme is as follows.

\begin{df}
 \label{DfAssocMatrices}
 Consider a set of non-zero $\{0,1\}$-matrices $A_0, \dots, A_d$ defined on $\pp \times \pp$.
 We call this set a \emph{d-class association scheme} if
 \begin{itemize}
  \item $A_0 + \ldots + A_d = J$,
  \item $A_0 = I$,
  \item all matrices $A_i$ are symmetric,
  \item there exist integers $p^k_{i,j}$ called the \emph{intersection numbers} such that for all $i,j$
  \[
   A_i A_j = \sum_{k=0}^d p^k_{i,j} A_k.
  \]
  \end{itemize}
\end{df}

If $\ma = \{A_0, \dots, A_d\}$ is an association scheme, then the subspace of $\RR^{\pp \times \pp}$ spanned by the elements of $\ma$, which we denote as $\RR[\ma]$, is closed under matrix multiplication.
Hence, it constitutes an algebra with respect to the ordinary matrix multiplication, called the \emph{Bose-Mesner algebra} of $\ma$.

\begin{res}
 Let $\ma = \{A_0, \dots, A_d\}$ be an association scheme on the set $\pp$.
 Then $\RR[\ma]$ admits a basis $E_0, \dots, E_d$ such that
 \begin{itemize}
  \item $E_0 = \frac 1 {|\pp|} J$,
  \item the $E_j$ matrices are idempotent and pairwise orthogonal, i.e.\ $E_i E_j = \delta_{i,j} E_i$,
  \item for all $i,j \in \{0,\dots,d\}$ there exists a number $p_i(j)$ such that $A_i E_j = p_i(j) E_j$.
 \end{itemize}
\end{res}

Note in particular that the column spaces of the $E_j$ matrices, which we denote by $V_j$, form an orthogonal decomposition of $\RR^{\pp}$ that diagonalises all $A_i$ matrices.
The matrix $\mathbf P$ defined by $\mathbf P(j,i) = p_i(j)$ is called the \emph{matrix of eigenvalues} of $\ma$.
Note that $\mathbf P$ is the transition matrix between the $A_i$-basis and the $E_j$-basis.
The matrix $\mathbf Q = |\pp| \mathbf P^{-1}$ is called the \emph{matrix of dual eigenvalues} of $\ma$.
It is well known that the first column of both $\mathbf P$ and $\mathbf Q$ is the all-one vector.
In particular, this implies the following result.

\begin{res}
 \label{Res:SumEigenvalues}
 For each $j > 0$, $\sum_i p_i(j) = 0$.
\end{res}

Define $n_i = p^0_{i,i}$ to be the \emph{valency} of the $R_i$-relation, and define $m_j = \rk(E_j)$, which is the dimension of $V_j$.

\begin{res}
 \label{ResIdentityEigenvalueMatrices}
 Let $\Delta_n$ and $\Delta_m$ be the diagonal matrices with diagonal $(n_0, \dots, n_d)$ and $(m_0, \dots, m_d)$ respectively.
 Then
 \[
  \Delta_m \mathbf P = \mathbf Q^\top \Delta_n.
 \]
\end{res}

Now we describe a way to calculate the eigenvalues of the association scheme.
Define for each $i \in \{0,\dots,d\}$ the $\{0,\dots,d\} \times \{0,\dots,d\}$ matrix $B_i$ given by $B_i(k,j) = p^k_{i,j}$.
This matrix is called the \emph{intersection matrix} of $R_i$.
There exists a basis $w_0, \dots, w_d$ of $\RR^{d+1}$, unique up to reordering and rescaling, that diagonalises all the intersection matrices.

\begin{res}\label{ResComputationEigenvaluesScheme}
 Let $w_0, \dots, w_d$ be the columns of $\mathbf Q$.
 Up to rescaling, $w_0, \dots, w_d$ is the unique basis of $\RR^{d+1}$ such that for all $i,j \in \{0,\dots,d\}$
 \[
  B_i w_j = p_i(j) w_j.
 \]
\end{res}

Hence, simultaneously diagonalising the intersection matrices yields the eigenvalues of $\ma$.

\bigskip

An association scheme $\ma = \{R_0, \dots, R_d\}$ on $\pp$ is called \emph{imprimitive} if the union of some relations of $\ma$ is an equivalence relation, and this union is not one of the trivial equivalence relations $R_0$ or $\pp \times \pp$; else it is called \emph{primitive}.
If $\pp' \subset \pp$ is one the equivalence classes of such a union in a primitive association scheme, then the relations $R_i \cap (\pp' \times \pp')$ on $\pp'$, after removing the empty relations, constitute an association scheme.
Such a scheme is called a \emph{subscheme} of $\ma$.

If $\ma$ and $\ma'$ are association schemes on $\pp$, and every relation of $\ma'$ is a union of relations of $\ma$, then we call $\ma'$ a \emph{fusion scheme} of $\ma$ and $\ma$ a \emph{fission scheme} of $\ma'$.

\bigskip

A particular class of association schemes that received a lot of attention are the 2-class association schemes, since these are equivalent to  \emph{strongly regular graphs}.
Note that if $A_0, A_1, A_2$ is an association scheme, then $A_0 = I$ and $A_2 = J - I  - A_1$.
Hence, all information is contained in the matrix $A_1$.
All parameters of the association scheme can be calculated from the parameters $n_1$, $p^1_{1,1}$, and $p^2_{1,1}$.
Therefore, instead of discussing the association scheme and its parameters, it is common to only refer to the graph induced by relation $R_1$, and denote this as a $\SRG(|\pp|,n_1,p^1_{1,1},p^2_{1,1})$.

\bigskip

It is often desirable to give a combinatorial interpretation of the eigenspaces of an association scheme.
We can relate sets to eigenspaces as follows.

\begin{df}
 Given a set $S \subseteq \pp$, define its \emph{characteristic vector} $\chi_S \in \RR^\pp$ as
 \[
  \chi_S(X) = \begin{cases}
   1 &\text{if } X \in S, \\
   0 &\text{otherwise.}
  \end{cases}
 \]
 Its \emph{dual degree set} is the set
 \[
  \dds(S) = \sett{j \in \{1,\dots,d\}}{ E_j \chi_S \neq \zero}.
 \]
\end{df}

In other words, $\dds(S)$ tells us in which eigenspaces $V_j$ the vector $\chi_S$ has a non-zero component (but note that $0 \notin \dds(S)$ by definition despite $E_0 \chi_S = \frac{|S|}{|\pp|} \one$).

Sets with small dual degree sets are especially interesting.
Such sets can be found e.g.\ using a weighted version Hoffman's ratio bound, see e.g. \cite[Theorem 2.4.2]{godsilmeagher}.
A matrix $A$ is \emph{compatible} with a graph $(\pp,R)$ if the rows and columns of $A$ are labelled by $\pp$ and $A(X,Y) \neq 0$ only if $(X,Y) \in R$.
A set $S$ is a \emph{clique}, or \emph{coclique}, in the graph $(\pp,R)$ if any pair, respectively no pair, of elements of $S$ occurs in $R$.

\begin{res}[Delsarte's ratio bound]
 \label{Res:Delsarte}
 Let $R_i$ be a single relation in an association scheme on $\pp$ with adjacency matrix $A_i$.
 Suppose that $A_i \one = n \one$ and that $\tau$ is the smallest eigenvalue of $A_i$.
 If $S$ is clique in $(\pp,R_i)$, then
 \[
  |S| \leq \frac n{-\tau} + 1.
 \]
 In case of equality, $\chi_S$ is orthogonal to the $\tau$-eigenspace of $A_i$.
\end{res}

\begin{res}[Weighted Hoffman's ratio bound]
 \label{Res:Hoffman}
 Let $(\pp,R)$ be a graph and $A$ a compatible symmetric matrix.
 Suppose that $A \one = n \one$ and that $\tau < n$ is the smallest eigenvalue of $A$.
 If $S$ is a coclique in $(\pp,R)$, then
 \[
  |S| \leq \frac{|\pp|}{\frac{n}{-\tau}+1}.
 \]
 In case of equality, $\chi_S$ is contained in the span of $\one$ and the $\tau$-eigenspace of $A$.
\end{res}

We call a clique or coclique attaining equality in the appropriate above bound a \emph{Delsarte clique} or \emph{weighted Hoffman coclique} respectively.
A weighted Hoffman coclique in case the matrix $A$ is just the adjacency matrix of $R$ is simply called a \emph{Hoffman coclique}.

\bigskip

An \emph{automorphism} of an association scheme is a permutation $g$ of the set $\pp$ such that $(X,Y) \in R_i \iff (X^g,Y^g) \in R_i$ for all $X,Y \in \pp$ and $i \in \{0,\dots,d\}$.
Automorphsims can be useful to find sets of characteristic vectors generating certain eigenspaces.

\begin{lm}
 \label{Lm:SpanningOrbit}
 Let $S$ be a non-empty subset of $\pp$.
 Let $G$ be a group of automorphisms such that for any $X,Y \in \pp$, the number $|\sett{g \in G}{X,Y \in S^g}|$ only depends on the relation containing $(X,Y)$.
 Then $\vspan{\chi_{S^g} \, \| \, g \in G} = \vspan{V_j \, \| \, j \in \dds(S) \cup \{0\}}$.
\end{lm}

\begin{proof}
 Consider the matrix $M \in \RR^{\pp \times G}$ where the column indexed by $g$ equals $\chi_{S^g}$.
 Then
 \[
  \vspan{\chi_{S^g} \, \| \, g \in G} = \ColSp(M) = \ColSp(M M^\top).
 \]
 Note that $M M^\top(X,Y) = |\sett{g \in G}{X,Y \in S^g}|$.
 By our assumptions, $M M^\top$ lies in the Bose-Mesner algebra of $\ma$.
 Therefore, there are coefficients $a_0, \dots, a_d$ with $M M^\top = \sum_j a_j E_j$ and $\ColSp(M M^\top) = \vspan{V_j \, \| \, a_j \neq 0}$.
 Since $G$ is a group of automorphisms, $S^g$ has the same dual degree set for all $g \in G$.
 It follows that $\sett{j}{a_j \neq 0} = \dds(S) \cup \{0\}$.
 %
\end{proof}

\section{The known association schemes on anisotropic points}
 \label{Sec:Known}

\subsection{Schurian schemes}
 \label{SubSecSchurianSchemes}

Let $G$ be a group acting on a finite set $\pp$.
Then $G$ naturally acts on $\pp \times \pp$, and the orbits of the latter group action are called the \emph{orbitals} of $G$ acting on $\pp$.
We say that $G$ acts \emph{generously transitively} on $\pp$ if $G$ acts transitively and the orbitals are symmetric.
It is well known that in this case the orbitals of the group action constitute an association scheme.
Such an association scheme is usually called \emph{Schurian}.

Given a (pseudo-)quadratic form $\kappa$ on a vector space $V$, an \emph{isometry} is a linear transformation $f \in \GL(V)$ such that $\kappa \circ f = \kappa$.
Isometries naturally act on the set of anisotropic points of a quadric or Hermitian variety.
This action is not necessarily transitive on the set of all anisotropic points.
Let $B$ denote the bilinear or sesquilinear form associated to $\kappa$.
If $q$ is odd and $\kappa$ is a quadratic form, then the action of the isometries has two orbits on the anisotropic points, namely the points with $\kappa(X) = S_q$ and the points with $\kappa(X) = \ns_q$.
If $q$ is even and $\kappa$ defines a parabolic quadric, then every isometry fixes the nucleus of the quadric, but acts transitively on the set of the other anisotropic points.
In the other cases, the isometries do act transitively on the anisotropic points.

The action of the isometries on one of its orbits of anisotropic points gives rise to an association scheme.
These association schemes were studied in \cite{bannaihaosong} for the non-degenerate quadrics, except for the case where $q$ and $n$ are even, and in \cite{wei83,bannaisonghaowei} for the non-degenerate Hermitian varieties.

Take one orbit of the action of the isometries on the anisotropic points.
Then the relation containing $(X,Y)$ depends on the value $\frac{B(x,y)B(y,x)}{\kappa(x) \kappa(y)}$, with $x$ and $y$ coordinate vectors of $X$ and $Y$ respectively.
The number of relations grows in $q$, which makes a general expression for the matrix of eigenvalues of such an association scheme quite involved.
Therefore, we will not include these matrices here.

\subsection{Geometrically defined schemes from quadrics in even characteristic}

Let $q$ be even and let $\pp$ be the set of anisotropic points of a quadric $\mq^\eps(n,q)$, excluding the nucleus $N$ in case $n$ is even.
Let $R_0$ as usual denote the identity relation.
If $X$ and $Y$ are distinct points of $\pp$, let them be in relation $R_1$, $R_2$, or $R_3$ if $|\vspan{X,Y} \cap \mq^\eps(n,q)|$ is 1, 2 or 0, respectively.
In case $n$ is even, split $R_1$ into two parts,
\begin{align*}
 R_{1a} = \sett{(X,Y) \in R_1}{ N \notin \vspan{X,Y} }, &&
 R_{1n} = \sett{(X,Y) \in R_1}{ N \in \vspan{X,Y} }.
\end{align*}
We note that in case $n=2$, relation $R_{1a}$ is empty.
In case $q=2$, lines contain 3 points, hence relations $R_{1n}$ and $R_2$ are empty.
\begin{itemize}
 \item If $n$ is odd, then the non-empty relations in $R_0,R_1,R_2,R_3$ constitute an association scheme.
 \item If $n$ is even, then the non-empty relations in $R_0,R_{1a},R_{1n},R_2,R_3$ constitute an association scheme.
\end{itemize}
This was first stated in \cite[Theorem 12.1.1]{bcn}, although the authors forgot to split relation $R_1$ into two parts in case $n>2$ is even.
This mistake was later corrected in \cite{vanhove20}. 
The matrices of eigenvalues and the dimension of the corresponding eigenspaces can be found in \cite{vanhove20} or \cite[\S 3.1.1]{brouwervanmaldeghem}.

In case $n=2m+1$ is odd, the matrices are given by

\begin{align*}
 \mathbf P & = \begin{pmatrix}
  1 & q^{2m}-1 & \frac12 q^m \left( q^m + \eps \right)(q-2) & \frac12 q^{m+1} \left(q^m- \eps \right) \\[0.5em]
  1 & \eps q^{m-1} - 1 & \frac12 \eps q^{m-1} (q+1)(q-2) & - \frac12 \eps q^m (q-1) \\[0.5em]
  1 & -\eps q^m - 1 & 0 & \eps q^m \\[0.5em]
  1 & \eps q^m -1 & -\eps q^m & 0
 \end{pmatrix} \\
 \mathbf Q &= \begin{pmatrix}
 1 & q^2 \frac{ q^{2m}-1 }{ q^2-1 } & \frac{q}{2(q+1)} \left( q^{m} - \eps \right) \left( q^{m+1} - \eps\right) & \frac{q-2}{2(q-1)}\left( q^m + \eps \right) \left( q^{m+1} - \eps \right) \\[0.5em]
 1 & \eps  q^2 \frac{q^{m-1} - \eps}{q^2-1} & -\frac12 \eps q \frac{q^{m+1} - \eps}{q+1} & \frac12 \eps (q-2) \frac{q^{m+1} - \eps}{q-1} \\[0.5em]
 1 & \eps q \frac{q^m - \eps}{q-1} & 0 & -\eps \frac{q^{m+1} - \eps}{q-1} \\[0.5em]
 1 & -\eps q \frac{q^m + \eps}{q+1} & \eps \frac{q^{m+1} - \eps}{q+1} & 0
 \end{pmatrix}
\end{align*}

\begin{rmk}
 The association schemes arising from $\mq^\eps(3,q)$ for $\eps = \pm 1$ have interesting alternative interpretations.
 Every plane in $\pg(3,q)$ intersects $\mq^\eps(3,q)$ either in a $\mq(2,q)$ or in a $\Pi_0 \mq^\eps(1,q)$.
 The polarity associated to the quadric maps the set of anisotropic points to the set of planes with non-degenerate intersections.
 Therefore, in an alternative interpretation of the association scheme, the vertices of the association scheme are the circles of the miquelian Möbius (if $\eps=-1$) or Minkowski (if $\eps=+1$) plane, and given a pair of circles $(c_1,c_2)$, $|c_1 \cap c_2|$ determines the relation containing $(c_1, c_2)$.
 In particular, this means that the association scheme arising from $\mq^-(3,q)$ can be seen as an association scheme on the Baer sublines of $\pg(1,q^2)$, and the association scheme arising from $\mq^+(3,q)$ can be seen as an association scheme on $\PGL(2,q)$.
 The eigenvalues of this association scheme on $\PGL(2,q)$ were also determined by Bannai \cite[p.\ 172]{bannai1991}.
\end{rmk}

In case $n=2m$ is even, the matrices are given by

\begin{align*}
 \mathbf P & = \begin{pmatrix}
  1 & q \left( q^{2m-2}-1 \right) & q-2 & \frac12 q^{2m-1}(q-2) & \frac 12 q^{2m} \\[0.5em]
  1 & -\left( q^{m-1} + 1 \right) (q-1) & q-2 & \frac12 q^{m-1}(q-2) & \frac12 q^m \\[0.5em]
  1 & \left( q^{m-1} - 1 \right) (q-1) & q-2 & -\frac12 q^{m-1}(q-2) & -\frac12 q^m \\[0.5em]
  1 & 0 & -1 & \frac12 q^m & -\frac12 q^m \\[0.5em]
  1 & 0 & -1 & -\frac12 q^m & \frac12 q^m
 \end{pmatrix} \\ \comments{
 \mathbf Q & = \begin{pmatrix}
  1 & \frac12 q \left( q^m + 1 \right) \frac{q^{m-1} - 1}{q-1} & \frac12 q \left( q^{m-1} + 1\right) \frac{q^m - 1}{q-1} & \frac 12 (q-2) \frac{q^{2m}-1}{q-1} & \frac 12 (q-2) \frac{q^{2m}-1}{q-1} \\[0.5em]
  1 & - \frac12 \left( q^m + 1 \right) & \frac12 \left( q^m - 1 \right) & 0 & 0 \\[0.5em]
  1 & \frac12 q \left( q^m + 1 \right) \frac{q^{m-1} - 1}{q-1} & \frac12 q \left(q^{m-1} + 1\right) \frac{q^m - 1}{q-1} & -\frac 12 \frac{q^{2m}-1}{q-1} & -\frac 12 \frac{q^{2m}-1}{q-1} \\[0.5em]
  1 & \frac 12 \frac{q^m + 1}{q^{m-1}} \frac{q^{m-1}-1}{q-1} & -\frac 12 \frac{q^{m-1}+1}{q^{m-1}} \frac{q^m - 1}{q-1} & \frac 1 {2q^{m-1}} \frac{q^{2m-1}-1}{q-1} & -\frac{1}{2q^{m-1}} \frac{q^{2m-1}-1}{q-1} \\[0.5em]
  1 & \frac12 \frac{q^m + 1}{q^{m-1}} \frac{q^{m-1} - 1}{q-1} & -\frac12 \frac{q^{m-1} + 1}{q^{m-1}} \frac{q^m - 1}{q-1} & -\frac 12 \frac{q-2}{q^m} \frac{q^{2m}-1}{q-1} & \frac{q-2}{2q^m} \frac{q^{2m}-1}{q-1}
 \end{pmatrix}
 \\ }
 \mathbf Q & =\begin{pmatrix}
  1 & \frac12 q \left( q^m + 1 \right) \theta_{m-2}(q) & \frac12 q \left( q^{m-1} + 1\right) \theta_{m-1}(q) & \frac 12 (q-2) \theta_{2m-1}(q) & \frac 12 (q-2) \theta_{2m-1}(q) \\[0.5em]
  1 & - \frac12 \left( q^m + 1 \right) & \frac12 \left( q^m - 1 \right) & 0 & 0 \\[0.5em]
  1 & \frac12 q \left( q^m + 1 \right) \theta_{m-2}(q) & \frac12 q \left(q^{m-1} + 1\right) \theta_{m-1}(q) & -\frac 12 \theta_{2m-1}(q) & -\frac 12 \theta_{2m-1}(q) \\[0.5em]
  1 & \frac 12 \frac{q^m + 1}{q^{m-1}} \theta_{m-2}(q) & -\frac 12 \frac{q^{m-1}+1}{q^{m-1}} \theta_{m-1}(q) & \frac 1 {2q^{m-1}} \theta_{2m-1}(q) & -\frac{1}{2q^{m-1}} \theta_{2m-1}(q) \\[0.5em]
  1 & \frac12 \frac{q^m + 1}{q^{m-1}} \theta_{m-2}(q) & -\frac12 \frac{q^{m-1} + 1}{q^{m-1}} \theta_{m-1}(q) & -\frac 12 \frac{q-2}{q^m} \theta_{2m-1}(q) & \frac12\frac{q-2}{q^m} \theta_{2m-1}(q)
 \end{pmatrix}
\end{align*}

\subsection{Geometrically defined schemes from quadrics in odd characteristic}
 \label{SubSecGeomSchemesOdd}

Let $q$ be odd, and let $\mq^\eps(n,q)$ be a non-degenerate quadric in $\pg(n,q)$ with quadratic form $\kappa$ and polarity $\perp$.
As discussed in \Cref{SubSecSchurianSchemes}, there are two classes of anisotropic points, depending on whether $\kappa(X) = S_q$ or $\kappa(X) = \ns_q$.
Let $\pp$ denote one of these classes, and say that two distinct points $X,Y$ of $\pp$ are in relation $R_1$, $R_2$, or $R_3$ when $| \vspan{X,Y} \cap \mq^\eps(n,q) |$ equals 1,2, or 0, respectively.


If $\eps = \pm 1$, the two classes of anisotropic points of $\mq^\eps(n,q)$ are interchangeable.
If $n=3$, then Brouwer, Cohen, and Neumaier \cite[Theorem 12.2.1]{bcn} proved that $(\pp,\{R_0,R_1,R_2,R_3\})$ constitutes an association scheme.
In case $\eps=-1$, the intersection matrices were determined by Fisher, Pentilla, Praeger, and Rolye \cite[p.\ 335]{FisherPenttilaPraegerRoyle}, and in case $\eps=+1$, the matrix of eigenvalues was determined by the first author \cite[\S 3]{adriaensen2022stability}.


If $\eps=0$, then $(\pp,\{R_0,R_1,R_2 \cup R_3\})$ constitutes an association scheme, as was observed by Wilbrink \cite[\S 7.D]{brouwervanlint}.
In this case, the two classes of anisotropic points are not interchangeable.
The class containing an anisotropic point $X$ depends on whether $X^\perp$ intersects $\mq(n,q)$ in a hyperbolic or elliptic quadric.
Suppose that $\pp$ consists of the points whose polar hyperplane intersects $\mq(n,q)$ in a $\mq^\eps(n-1,q)$.
Then $R_1$ determines a 
\[
 \SRG \left( \frac 12 q^\frac n2 \left( q^\frac n2 + \eps \right), \left( q^\frac n2 - \eps \right) \left( q^{\frac n2 - 1} + \eps \right), 2(q^{n-2} - 1) + \eps q^{\frac n2 - 1}(q-1), 2 q^{\frac n2 - 1} \left( q^{\frac n2 - 1} + \eps \right) \right).
\]
The spectrum of this graph is given in \Cref{Tab:SpecSRGParabolic}. 
\begin{table}[h!]
 \centering
 \begin{tabular}{c || c | c | c}
  Eigenvalue & $\left( q^\frac n2 - \eps \right) \left( q^{\frac n2 - 1} + \eps \right)$ & $-\eps q^{\frac n2 -1} - 1$ & $\eps (q-2)q^{\frac n2 - 1} - 1$ \\ \hline
  Multiplicity & 1 & $\frac {q-2}2 \theta_{n-1}(q)$ & $\frac q 2  \left( q^\frac n2 -\eps \right) \frac{q^{\frac n2 - 1}+\eps}{q-1}$
 \end{tabular}
 \caption{Spectrum of a strongly regular graph related to the parabolic quadrics.}
 \label{Tab:SpecSRGParabolic}
\end{table}

By applying the polarity, we can also interpret this association scheme as being defined on the hyperplanes intersecting $\mq(n,q)$ in a $\mq^\eps(n-1,q)$, where two distinct hyperplanes $\pi$ and $\rho$ are in relation $R_1$ if and only if $\pi \cap \rho \cap \mq(n,q)$ is a degenerate quadric.
This also yields a strongly regular graph in case $q$ is even, where the same formulae for the parameters apply.
More information about these graphs can be found in \cite[\S 3.1.4]{brouwervanmaldeghem}.

\subsection{Geometrically defined schemes from Hermitian varieties}
 \label{Sec:NU}

Every line of $\pg(n,q^2)$ is either contained in the Hermitian variety $\mh(n,q^2)$ or intersects it in 1 or $q+1$ points.
We say that two distinct anisotropic points $X$ and $Y$ are in relation $R_1$ or $R_2$ if $\vspan{X,Y}$ intersects $\mh(n,q^2)$ in 1 or $q+1$ points, respectively.
These relations, together with the identity, define an association scheme.
The strongly regular graph defined by $R_1$ is denoted as $NU(n+1,q^2)$.
This scheme was first described by Chakravarti \cite{chakravarti71} for $n = 2,3$.
The construction in general dimension can be found e.g.\ in \cite[\S 3.1.6]{brouwervanmaldeghem}.
Define $\eps = (-1)^n$.
The matrices of this scheme are given by
\begin{align*}
 \mathbf P &= \begin{pmatrix}
  1 & \left( q^n - \eps \right) \left( q^{n-1} + \eps \right) & q^{n-1} \frac{q^n - \eps}{q+1} (q^2-q-1) \\[0.5em]
  1 & - \eps q^{n-1} - 1 & \eps q^{n-1} \\[0.5em]
  1 & \eps q^{n-2} (q^2-q-1) - 1 & -\eps q^{n-2} (q^2-q-1)
 \end{pmatrix} \\
 \mathbf Q & = \begin{pmatrix}
  1 & \frac{q^2-q-1}{(q+1)(q^2-1)}\left( q^{n+1} + \eps \right) \left( q^n - \eps \right) & \frac{q^3}{(q+1)(q^2-1)} \left( q^{n-1} + \eps \right)\left( q^n - \eps \right) \\[0.5em]
  1 & -\eps \frac{q^2-q-1}{(q+1)(q^2-1)}\left( q^{n+1} + \eps \right) & \eps \frac{q^2-q-1}{(q+1)(q^2-1)}\left( q^{n+1} + \eps \right) - 1 \\[0.5em]
  1 & \eps q^2 \frac{q^{n-1} + \eps }{q^2-1} - 1 & - \eps q^2 \frac{q^{n-1} + \eps }{q^2-1}
 \end{pmatrix}
\end{align*}

If $X$ and $Y$ are anisotropic points and $X \perp Y$, then $(X,Y) \in R_2$.
This allows us to split $R_2$ into the relations $R_{2\perp} = \sett{(X,Y) \in R_2}{X \perp Y}$ and $R_{2\not\perp} = R_2 \setminus R_{2\perp}$.
Gordon and Levingston \cite[p.\ 264]{gordonlevingston} observed that in case $n=2$, $R_0,R_1,R_{2\perp},R_{2\not\perp}$ constitutes an association scheme.
We will prove in Section \ref{SecHerm} that this holds in general dimension.
In particular this forms a fission scheme of the association scheme related to the stongly regular graph $NU(n+1,q^2)$.

\begin{prop} \label{Prop:FissionHerm}
 For general values of $n \geq 2$, the above defined relations $R_0, R_1, R_{2\perp}, R_{2 \not\perp}$ constitute an association scheme on the anisotropic points of $\mh(n,q^2)$.
\end{prop}

\begin{rmk}
 \label{Rmk:HermOrtho}
 \begin{enumerate}
  \item \label{Rmk:HermOrtho:Q=2} One should be cautious in case $q=2$.
 In that case $R_{2 \not \perp} = \emptyset$.
 In case $n = 2$, the graph corresponding to relation $R_{2\perp}$ is not connected, but consists of 4 disjoint triangles.
  \item \label{Rmak:HermOrtho:Brouwer} Ferdinand Ihringer privately communicated to us that Andries Brouwer independently proved \Cref{Prop:FissionHerm} in a work in progress.
 \end{enumerate}
\end{rmk}

We postpone the proof of this proposition to \Cref{SecHerm}.

\section{Geometrically defined schemes from quadrics in odd dimension and odd characteristic}
 \label{Sec:NewScheme}

Let $q$ be an odd prime power.
Let $n$ be odd, and let $\mq$ denote $\mq^\eps(n,q)$, with $\eps \in \set{\pm 1}$.
Let $B$, $\kappa$, and $\perp$ denote respectively the bilinear form, quadratic form, and polarity associated to $\mq$.
Let $\pp$ denote the set of anisotropic points.
We partition $\pp \times \pp$ into relations $R_i$, where $R_0 = \sett{(X,X)}{X \in \pp}$ is the identity relation, and the relation between two distinct points $X$ and $Y$ is based on how $\vspan{X,Y}$ intersects $\mq$, and on $\kappa(X) \cdot \kappa(Y)$.
The definition of the relations can be found in Table \ref{tab:rel}.

\begin{table}[!ht]
    \centering
    \begin{tabular}{c||c|c}
      & $| \vspan{X,Y} \cap \mq |$ & $\kappa(X) \cdot \kappa(Y)$ \\ \hline \hline
     $R_1$ & 1 & $S_q$ \\
     $R_2$ & 2 & $S_q$ \\
     $R_3$ & 0 & $S_q$ \\ [.2em] \hline & &\\ [-1em]
     $R_4$ & 2 & $\ns_q$ \\
     $R_5$ & 0 & $\ns_q$
    \end{tabular}
    \caption{The relations of the association scheme.}
    \label{tab:rel}
\end{table}

Note that the relation given by $|\vspan{X,Y} \cap \mq| = 1$ and $\kappa(X) \times \kappa(Y) = \ns_q$ is missing.
This is due to the fact that this relation is always empty, or in other words due to the fact that 
\(
 |\vspan{X,Y} \cap \mq| = 1 \implies \kappa(X) \cdot \kappa(Y) = S_q
\). Note that these relations are symmetric.

Our goal is to prove that the relations from Table \ref{tab:rel} together with $R_0$ constitute an association scheme on $\pp$, and to find the matrices of eigenvalues and dual eigenvalues.

Recall that
for $i \in \{0, \ldots, 5\}$ and $X \in \pp$
\[
 R_i(X) = \sett{Y \in \pp}{(X,Y) \in R_i}.
\]

\begin{df}
    Given $(X,Y)\in R_{k}\subseteq\pp\times\pp$ we denote the size of $R_i(X) \cap R_j(Y)$ by $p^{k}_{i,j}(X,Y)$.
\end{df}

In a series of lemmas we will derive the values of all these $p^{k}_{i,j}(X,Y)$. It will be shown that they are independent from the chosen pair $(X,Y)$, so we can omit $(X,Y)$ in the notation $p^{k}_{i,j}(X,Y)$. To make to lemmas and their proofs less notation-heavy, we will already use $p^{k}_{i,j}$. Since the relations $R_{k}$ are symmetric we know that $p^{k}_{i,j}(X,Y) = p^{k}_{j,i}(Y,X)$.

\begin{rmk}\label{Rem:OrthSchemeForq=3}
 If $q=3$, then relation $R_2$ is empty. For this case the proofs of this section show that $\{R_0,R_1,R_3,R_4,R_5\}$ forms a 4-class association scheme. Note that in particular all values $p^{k}_{i,j}$ with $2\in\{i,j\}$ and $k\neq2$ have a factor $q-3$. Also, the dimension of the eigenspace $V_4$ for the 5-class association scheme has a factor $q-3$. So, the $\mathbf{P}$ and $\mathbf{Q}$ matrix for the 4-class association scheme that we find for $q=3$, are found by removing the column and row corresponding to the relation $R_2$ and the eigenspace $V_4$ in the matrices that we obtain in \Cref{SecPQmatrix}.
\end{rmk}

\subsection{Some general counting arguments}

\begin{df}
 For each $X \in \pp$, define $n_i(X) = |R_i(X)|$.
\end{df}

The number $n_i(X)$ is independent of $X$.
This follows for example from the fact that the projective orthogonal group stabilising $\mq$ acts transitively on $\pp$ and respects the relations $R_i$.
Hence, we denote these numbers just as $n_i$.

\begin{lm}
 \label{LmValencies}
 \begin{align*}
  n_1 &=  q^{n-1}-1\\
  n_2 &=  \frac 14 q^{\frac{n-1}2} \left(q^{\frac{n-1}2} + \eps \right) (q-3)\\
  n_3 &=  \frac 14 q^{\frac{n-1}2} \left(q^{\frac{n-1}2} - \eps \right) (q-1)\\
  n_4 &=  \frac 14 q^{\frac{n-1}2} \left(q^{\frac{n-1}2} + \eps \right) (q-1)\\
  n_5 &=  \frac 14 q^{\frac{n-1}2} \left(q^{\frac{n-1}2} - \eps \right) (q+1)
 \end{align*}
\end{lm}

\begin{proof}
 Take a point $X \in \pp$.
 A line through $X$ is tangent if and only if it intersects $X^\perp$ in a point of $\mq$.
 Since $\mq$ intersects $X^\perp$ in a non-singular parabolic quadric, there are $|\mq(n-1,q)| = \theta_{n-2}(q)$ tangent lines through $X$.
 The number of secant lines through $X$ then equals
 \[
  \frac 12 \left(|\mq| - \theta_{n-2}(q)\right) = \frac 12 \left(q^{n-1} + \eps q^{\frac{n-1}2}\right),
 \]
 where we used Result \ref{ResQuadricSize}.
 The number of passant lines through $X$ is therefore
 \[
  \theta_{n-1}(q) - \theta_{n-2}(q) - \frac 12 \left(q^{n-1} + \eps q^{\frac{n-1}2} \right)
  = \frac 12 \left(q^{n-1} - \eps q^{\frac{n-1}2} \right).
 \]
 
 Now take a line $\ell$ through $X$.
 If $\ell$ is tangent, then it contains $q-1$ points of $\pp$ in relation $R_1$ with respect to $X$.
 If $\ell$ is secant, then it contains $\frac{q-1}2$ points of $\pp$ of both types, hence $\frac{q-3}2$ points in relation $R_2$, and $\frac{q-1}2$ points in relation $R_4$ with respect to $X$.
 If $\ell$ is a passant line, then in contains $\frac{q+1}2$ points of each type, hence, $\frac{q-1}2$ points in relation $R_3$ with respect to $X$ and $\frac{q+1}2$ points in relation $R_5$ with respect to $X$.
 
 It is now easy to calculate the $n_i$ numbers.
\end{proof}

This is especially useful since the following identity holds. Note that is identity is well-known for association schemes, but we yet have to show that $\{R_0,\dots,R_5\}$ constitutes an association scheme.

\begin{lm}
 \label{LmSumValency}
 Take $(X,Y) \in R_k$, then for any $i$
 \[
  n_i = \sum_{j=0}^5 p_{i,j}^k(X,Y) = \sum_{j=0}^5 p_{j,i}^k(X,Y).
 \]
\end{lm}
\begin{proof}
 We prove the first equality, the second one is completely analogous.
 \[
  n_i = |R_i(X)|
  = \left| \bigcup_{j=0}^5 R_i(X) \cap R_j(Y) \right|
  = \sum_{j=0}^5 p_{i,j}^k (X,Y).
  \qedhere
 \]
\end{proof}

A non-isotropic plane can intersect $\mq$ in 5 possible ways.
We will introduce some terminology which will be convenient in the arguments below.

\begin{df}
 We say that a plane $\pi$ in $\pg(n,q)$ is of type $t_i$ if the plane intersects $\mq$ exactly in a point $Q \in \mq$ and $i$ lines through $Q$, $i \in \{0,1,2\}$.
 If $\pi$ intersects $\mq$ in a conic $C$, then the external points in $\pi$ to $C$ (i.e.~the anisotropic points that are on tangent lines, as opposed to the internal points, which are not on tangent lines) are either the points of square type, in which case we say that $\pi$ is of type $t_{S_q}$, or the points of non-square type, in which case we say that $\pi$ is of type $t_{\ns_q}$.
\end{df}

\begin{lm}
 \label{LmPlanesThroughLines}
 Suppose that $\ell$ is a line, which isn't totally isotropic.
 In case that $\ell$ is tangent to $\mq$ suppose that the anisotropic points of $\ell$ are of type $s \in \{S_q, \ns_q\}$.
 \Cref{TabPlanesThroughLines} shows the number of planes of each type through $\ell$.
 \begin{table}[ht!]
  \small
  \centering
  \begin{tabular}{c|| c c c c}
   & $|\ell \cap \mq| = 1$ & $|\ell \cap \mq| = 1$ & $|\ell \cap \mq| = 2$ & $|\ell \cap \mq| = 0$\\
   & $s=S_q$ & $s=\ns_q$ \\ \hline \hline \\ [-1em]
   $t_0$ & $\frac{1}{2}\left(q^{n-3}-\varepsilon q^{\frac{n-3}{2}}\right)$ & $\frac{1}{2}\left(q^{n-3}-\varepsilon q^{\frac{n-3}{2}}\right)$ & 0 & $\frac 1 {q-1} \left( q^{\frac{n-1}2} + \eps \right)\left( q^{\frac{n-3}2} - \eps \right)$ \\
   $t_1$ & $\theta_{n-4}(q)$ & $\theta_{n-4}(q)$ & 0 & 0 \\
   $t_2$ & $\frac{1}{2}\left(q^{n-3}+\varepsilon q^{\frac{n-3}{2}}\right)$ & $\frac{1}{2}\left(q^{n-3}+\varepsilon q^{\frac{n-3}{2}}\right)$ & $\frac 1 {q-1} \left( q^{\frac{n-1}2} - \eps \right)\left( q^{\frac{n-3}2} + \eps \right)$ & 0 \\ [1em] \hline \\ [-1em]
   $t_{S_q}$ & $q^{n-2}$ & 0 & $\frac{1}{2}\left(q^{n-2}-\eps q^{\frac{n-3}{2}}\right)$ & $\frac{1}{2}\left(q^{n-2}+\eps q^{\frac{n-3}{2}}\right)$ \\
   $t_{\ns_q}$ & 0 & $q^{n-2}$ & $\frac{1}{2}\left(q^{n-2}-\eps q^{\frac{n-3}{2}}\right)$ & $\frac{1}{2}\left(q^{n-2}+\eps q^{\frac{n-3}{2}}\right)$
  \end{tabular}
  \caption{The number of planes of a certain type through a non-isotropic line}
  \label{TabPlanesThroughLines}
 \end{table}
\end{lm}

\begin{proof}
 First suppose that $\ell$ intersects $\mq$ in a unique point $P$.
 Then $\ell \subseteq P^\perp$.
 If a plane $\pi$ through $\ell$ intersects $\mq$ in a singular plane section, then $P$ must be a singular point of this plane section, since it lies on a tangent line.
 This happens if and only if $\pi \subseteq P^\perp$.
 Since there are $\theta_{n-2}(q)$ planes through $\ell$, of which $\theta_{n-3}(q)$ lie in $P^\perp$, there are $\theta_{n-2}(q)-\theta_{n-3}(q) = q^{n-2}$ planes $\pi$ through $\ell$ intersecting $\mq$ in a conic.
 Since the non-isotropic points of $\ell$ lie on a tangent line in such a plane $\pi$, the type of $\pi$ is $t_s$.
 
 Now let $\sigma$ be a hyperplane in $P^\perp$, not through $P$.
 Then $\sigma$ intersects $\mq$ in a quadric isomorphic to $\mq^\eps(n-2,q)$.
 Let $T$ denote $\ell \cap \sigma$.
 Then there is a one-to-one correspondence between the planes of type $t_i$ through $\ell$ and the lines in $\sigma$ through $T$ that intersect $\mq$ in $i$ points.
 As in the proof of \Cref{LmValencies}, one can calculate that $\sigma$ contains $\theta_{n-4}(q)$ tangent lines, $\frac12 \left( q^{n-3} + \eps q^{\frac{n-3}2} \right)$ secant lines, and $\frac 12 \left( q^{n-3} - \eps q^{\frac{n-3}2} \right)$ passant lines through $T$.
 
 Now suppose that $\ell$ intersects $\mq$ in $\delta+1$ points, with $\delta \in \{\pm 1\}$.
 Then $\ell \cap \mq$ is isomorphic to $\mq^\delta(1,q)$.
 Hence, $\ell^\perp \cap \mq$ is isomorphic to $\mq^{\eps \delta}(n-2,q)$.
 Moreover, every plane $\pi$ through $\ell$ intersects $\ell^\perp$ in a unique point $T$.
 The intersection of $\pi$ and $\mq$ is singular if and only if $T \in \mq \cap \ell^\perp$.
 Hence, there are $|\mq^{\eps \delta}(n-2,q)| =\theta_{n-3}(q)+\eps\delta q^{\frac{n-3}{2}}$ singular planes through $\ell$, necessarily of type $t_{\delta+1}$. 
 Lastly, consider the case where $T \notin \mq$.
 Since $T \in \ell^\perp$, we know that $\ell \subset T^\perp$.
 Moreover, $T \notin T^\perp$ since $T \notin \mq$, thus $\ell = T^\perp \cap \pi$.
 Therefore, the tangent lines through $T$ in $\pi$ are exactly the lines through $T$ in $\pi$ that intersect $\ell$ in a point of $\mq$.
 Thus, $\pi$ is of type $t_{\kappa(T)}$ if $\delta=1$ and $t_{\ns_q \cdot \kappa(T)}$ if $\delta=-1$.
 Since $\ell^\perp$ intersects $\mq$ in a quadric isomorphic to $\mq^{\eps \delta}(n-2,q)$, it contains equally many non-isotropic points of both quadratic types.
 Therefore the number of planes of type $t_{S_q}$ through $\ell$ equals the number of planes of type $t_{\ns_q}$ through $\ell$.
 Since there are $\theta_{n-2}(q)$ planes through $\ell$, this number equals
 \[
  \frac12 \left( \theta_{n-2}(q) - \left( \theta_{n-3}(q)+\eps\delta q^{\frac{n-3}{2}}\right) \right) = \frac 12 \left( q^{n-2} - \eps \delta q^{\frac{n-3}2} \right). \qedhere
 \]
\end{proof}

Next we prove that $p^k_{i,j}(X,Y) = p^k_{j,i}(X,Y)$ whenever $k \geq 2$.
This follows from the next lemma.

\begin{lm}
 \label{lmHulp}
 Suppose that $X$ and $Y$ are different points outside of $\mq$, such that $\vspan{X,Y}$ is not a tangent line.
 Then there exists a collineation stabilising $\mq$ that swaps $X$ and $Y$.
\end{lm}

\begin{proof}
 Since the line $\ell = \vspan{X,Y}$ is non-singular, $\ell$ is disjoint to $\ell^\perp$.
 Let $e_0, \dots, e_n$ denote the standard basis of $\FF_q^{n+1}$.
 We may suppose that $\ell = \vspan{e_0, e_1}$ and $\ell^\perp = \vspan{e_2, \dots, e_n}$.
 Thus, the quadratic form $\kappa$ defining $\mq$ is of the form $\kappa(X) = f(X_0,X_1) + g(X_2, \dots, X_n)$.
 It suffices to prove that there exists a $\varphi \in \GL(2,q)$ such that $f \circ \varphi = \alpha f$ for some $\alpha \in \FF_q^*$ and such that the collineation on $\ell$ induced by $\varphi$ swaps $X$ and $Y$.
 Indeed, since for each constant $\alpha \in \FF_q^*$ there exists a $\psi \in \GL(n-1,q)$ such that $g \circ \psi = \alpha g$, the collineation induced by $(X_0, \dots, X_n) \mapsto (\varphi(X_0,X_1), \psi(X_2, \dots, X_n))$ then satisfies the properties of the lemma.
 
 We may assume that $f(X_0,X_1) = X_0^2 - \nu X_1^2$ for some $\nu \in \FF_q^*$, with $\nu$ a square if and only if $\mq$ is hyperbolic, and $X = \vspan{e_0}$.
 If $Y = \vspan{e_1}$, then we can use $\varphi(X_0,X_1) = \left( X_1, \frac{X_0}\nu \right)$.
 Otherwise, $Y = \vspan{e_0 + \alpha e_1}$ for some $\alpha \in \FF_q^*$.
 Then we can use $\varphi(X_0,X_1) = (X_0 - \alpha \nu X_1, \alpha X_0 - X_1)$.
 Note that in this case $\varphi$ being invertible is equivalent to $Y = \vspan{e_0 + \alpha e_1} \notin \mq$.
\end{proof}

\subsection{Computation of the intersection numbers}

Throughout this subsection, assume that $X$ and $Y$ are distinct points of $\pp$ and let $\ell$ denote the line joining $X$ and $Y$.

\begin{lm}
 \label{LmIntersectionNumbersEqualToZero}
    For $(X,Y)\in R_{i}$ we have the following equalities:
    \begin{align*}
        0&=p^{i}_{1,4}=p^{i}_{1,5}=p^{i}_{2,4}=p^{i}_{2,5}=p^{i}_{3,4}=p^{i}_{3,5}&&\text{ for } i=1,2,3\quad \text{and}\\
        0&=p^{i}_{1,1}=p^{i}_{1,2}=p^{i}_{1,3}=p^{i}_{2,2}=p^{i}_{2,3}=p^{i}_{3,3}=p^{i}_{4,4}=p^{i}_{4,5}=p^{i}_{5,5}\quad&&\text{ for } i=4,5.
    \end{align*}
\end{lm}
\begin{proof}
    We know that $\kappa(X)\cdot\kappa(Y)=(\kappa(X)\cdot \kappa(Z))\cdot(\kappa(Y)\cdot\kappa(Z))$. The statements then immediately follow from the observations $S_q \cdot S_q = \ns_q \cdot \ns_q = S_q$, and $S_q \cdot \ns_q = \ns_q$.
\end{proof}

\begin{lm}[{\cite[\S 12.2, p.\ 380]{bcn}}]
 \label{lem:p122}
    Let $C$ be a conic in $\pg(2,q)$ with corresponding quadratic form $\kappa$, and let $X,Y\notin C$ be distinct points in $\pg(2,q)$ such that $\vspan{X,Y}$ is a tangent to $C$. The number of points $Z$ such that both $\vspan{X,Z}$ and $\vspan{Y,Z}$ are secant lines and $\kappa(X)\cdot\kappa(Z)=S_{q}=\kappa(Y)\cdot\kappa(Z)$ equals $\frac{(q-3)(q-5)}{8}$.
\end{lm}

\begin{lm}
 \label{lem:R1relations}
    For $(X,Y)\in R_{1}$ we have the following equalities.
    {\allowdisplaybreaks \begin{align*}
        p^{1}_{1,1}&=2(q^{n-2}-1)\\
        p^{1}_{1,2}&=\frac{1}{2}q^{n-2}(q-3)\\
        p^{1}_{1,3}&=\frac{1}{2}q^{n-2}(q-1)\\
        p^{1}_{2,2}&=\frac{1}{8}q^{\frac{n-1}{2}}\left(q^{\frac{n-1}{2}}-3q^{\frac{n-3}{2}}+2\eps\right)(q-3)\\
        p^{1}_{2,3}&=\frac{1}{8}q^{n-2}(q-1)(q-3)\\
        p^{1}_{3,3}&=\frac{1}{8}q^{\frac{n-1}{2}}\left(q^{\frac{n-1}{2}}-q^{\frac{n-3}{2}}-2\eps\right)(q-1)\\
        p^{1}_{4,4}&=\frac{1}{8}q^{\frac{n-1}{2}}\left(q^{\frac{n-1}{2}}-q^{\frac{n-3}{2}}+2\eps\right)(q-1)\\
        p^{1}_{4,5}&=\frac{1}{8}q^{n-2}(q+1)(q-1)\\
        p^{1}_{5,5}&=\frac{1}{8}q^{\frac{n-1}{2}}\left(q^{\frac{n-1}{2}}+q^{\frac{n-3}{2}}-2\eps\right)(q+1)
    \end{align*}}
\end{lm}

\begin{proof}
 From \Cref{LmPlanesThroughLines}, we know the number of planes of each type through $\ell$.
 We will compute how much each of these planes contributes to the numbers $p^1_{1,1}$, $p^1_{1,2}$, $p^1_{2,2}$, and $p^1_{4,4}$.
 Then we can compute the remaining intersection numbers $p^1_{i,j}$ using \Cref{LmValencies} and \Cref{LmSumValency}.
 First note that $\ell$ contains $q-2$ anisotropic points distinct from $X$ and $Y$, all of which are in relation $R_1$ with respect to $X$ and $Y$.
 Let $P$ denote $\ell \cap \mq$.
 Take a plane $\pi$ through $\ell$.
 \begin{itemize}
  \item If $\pi$ is of type $t_0$, then any point of $\pi \setminus \ell$ is in relation $R_3$ or $R_5$ with respect to $X$ and $Y$.
  \item If $\pi$ is of type $t_1$, then there are $q(q-1)$ anisotropic points in $\pi \setminus \ell$, all in relation $R_1$ with respect to $X$ and $Y$.
  \item If $\pi$ is of type $t_2$, all anisotropic points of $\pi \setminus \ell$ lie on a secant line through $X$ and on a secant line through $Y$.
  Take a secant line $\ell'$ in $\pi$.
  Then $\ell'$ contains $\frac{q-1}2$ points of each quadratic type.
  Since each anisotropic point of $\pi$ lies on a tangent line through $P$ that intersects $\ell'$ in an anisotropic point, and all anisotropic points on a tangent line are of the same type, $\pi$ contains $q \frac{q-1}2$ anisotropic points of each type.
  Hence, there are $q\frac{q-3}2$ points in $R_2(X) \cap R_2(Y)$, and $q \frac{q-1}2$ points in $R_4(X) \cap R_4(Y)$.
  \item If $\pi$ is a non-singular plane, then it intersects $\mq$ in a conic, and there are unique tangent lines $\ell_X, \ell_Y \neq \ell$ going through $X$ and $Y$, respectively.
  There is only point of $\pi \setminus \ell$ in $R_1(X) \cap R_1(Y)$, namely $\ell_X \cap \ell_Y$.
  
  Now we count the number of points of $\pi$ in $R_1(X) \cap R_2(Y)$.
  These points lie necessarily on $\ell_X$.
  There are $\frac{q-1}2$ secant lines through $Y$ in $\pi$.
  Of these secant lines, $\vspan{Y,\ell_X \cap \mq}$ is the unique one intersecting $\ell_X$ in an isotropic point.
  Thus $\pi$ contains $\frac{q-3}2$ points of $R_1(X) \cap R_2(Y)$.
  
  Since $\pi$ contains $\frac{q-1}2$ secant lines through $X$ and equally many through $Y$, $\pi$ contains $\left(\frac{q-1}2 \right)^2$ points lying on a secant line through $X$ and a secant line through $Y$.
  Note that the conic $\pi \cap \mq$ contains $3$ points lying on a tangent line through $X$ or $Y$, so $q-2$ points lying on both a secant line through $X$ and $Y$.
  Therefore, $\pi$ contains $\left(\frac{q-1}2 \right)^2 - (q-2)$ points of $(R_2(X) \cap R_2(Y)) \cup (R_4(X) \cap R_4(Y))$.
  By \Cref{lem:p122}, $\frac{(q-3)(q-5)}8$ of these are in $R_2(X) \cap R_2(Y)$, which leaves
  \[
   \left(\frac{q-1}2 \right)^2 - (q-2) - \frac{(q-3)(q-5)}8 = \frac{(q-1)(q-3)}8
  \]
  points in $R_4(X) \cap R_4(Y)$.
 \end{itemize}
 
 This yields the following equalities.
 \begin{align*}
        p^{1}_{1,1}&=q-2+(q^{2}-q)\theta_{n-4}(q)+q^{n-2}\:,\\
        p^{1}_{1,2}&=q^{n-2}\left(\frac{q-3}{2}\right)\:,\\
        p^{1}_{2,2}&=\frac{1}{2}\left(q^{n-3}+\varepsilon q^{\frac{n-3}{2}}\right)\frac{q(q-3)}{2}+q^{n-2}\frac{(q-3)(q-5)}{8}\:,\\
        p^1_{4,4} &= \frac{1}{2}\left(q^{n-3}+\varepsilon q^{\frac{n-3}{2}}\right) \frac{q(q-1)}2 + q^{n-2} \frac{(q-1)(q-3)}8\:.
 \end{align*}
 Using \Cref{LmValencies}, \Cref{LmSumValency} and \Cref{LmIntersectionNumbersEqualToZero}, we can compute the remaining intersection numbers.
\end{proof}

The next lemma generalises \cite[Theorem 12.2.1]{bcn}.

\begin{lm}\label{lem:p222}
    For $(X,Y)\in R_{2}$ we have
    \begin{align*}
        p^{2}_{2,2}&=\frac{1}{8}q^{\frac{n-3}{2}}\left(\left(q^{\frac{n-1}{2}}-\eps\right)(q-3)^{2}+4\eps q(q-5)\right)\:,
    \end{align*}
    and for $(X,Y)\in R_{3}$ we have
    \begin{align*}
        p^{3}_{2,2}&=\frac{1}{8}q^{\frac{n-3}{2}}\left(q^{\frac{n-1}{2}}+\eps\right)(q-3)^{2}\:.
    \end{align*}
\end{lm}

\begin{proof}
    Let $\ell$ be the line $\vspan{X,Y}$, and let $\eps'$ be such that $\ell\cap\mq$ is a $\mq^{\eps'}(1,q)$. Then we know that $\ell^{\perp}$ is disjoint to $\ell$ and that $\ell^{\perp}\cap\mq$ is a quadric $\mq^{\eps\eps'}(n-2,q)$. Let $W$, $V_{1}$ and $V_{2}$ be the underlying vector spaces of $\pg(n,q)$, $\ell$ and $\ell^{\perp}$, respectively. We know $W=V_{1}\oplus V_{2}$.
    \par Let $B$ be the bilinear form introduced in the introduction, connected to $\kappa$, and let $x$ and $y$ be vector representatives of the points $X$ and $Y$, respectively.
    For a point $Z$ with $z$ as a vector representative, the lines $\vspan{X,Z}$ and $\vspan{Y,Z}$ are secant lines if and only if $(B(x,z))^{2}-\kappa(x)\kappa(z)\in S_{q}$ and $(B(y,z))^{2}-\kappa(y)\kappa(z)\in S_{q}$. The vector $z$ can uniquely be written as $v_{1}+v_{2}$ with $v_{1}\in V_{1}$ and $v_{2}\in V_{2}$. Since $V_{1}=V_{2}^{\perp}$, we have that $B(x,z)=B(x,v_{1})$ and $B(y,z)=B(y,v_{1})$. Also, note that $\kappa(z)=\kappa(v_{1})+\kappa(v_{2})$.
    \par Now we define the following set:
    \begin{align*}
        \mathcal{S}&=\{(v_{1},v_{2})\in V_{1}\times V_{2}\, \| \, \forall w\in\{x,y\}:(B(w,v_{1}))^{2}-\kappa(w)(\kappa(v_{1})+\kappa(v_{2}))\in S_{q}\text{ and }\\&\qquad\qquad\qquad\qquad\qquad\qquad \kappa(x)(\kappa(v_{1})+\kappa(v_{2}))\in S_{q}\}\:.
    \end{align*}
    Note that $\kappa(x)(\kappa(v_{1})+\kappa(v_{2}))\in S_{q}$ if and only if $\kappa(y)(\kappa(v_{1})+\kappa(v_{2}))\in S_{q}$. Since $(0,0)\notin\mathcal{S}$ we know that $|\mathcal{S}|=(q-1)p^{\theta}_{2,2}$ with $\theta=\frac{5-\eps'}{2}$. We will determine $|\mathcal{S}|$.
    Define the following set for all $\alpha\in\FF_{q}^{*}$:
    \begin{align*}
        T_{\alpha}=\{v_{1}\in V_{1}\, \| \, \forall w\in\{x,y\}:(B(w,v_{1}))^{2}-\alpha\kappa(w)\in S_{q}\}\:.
    \end{align*}
    Using this notation, we find that
    { \allowdisplaybreaks \begin{align*}
        |\mathcal{S}|&=\sum_{a\in\FF_{q}}|\{(v_{1},v_{2})\in V_{1}\times V_{2}\, \| \,\forall w\in\{x,y\}:(B(w,v_{1}))^{2}-a\kappa(w)\in S_{q},\ a\kappa(x)\in S_{q}\text{ and }\\&\qquad\qquad\qquad\qquad\qquad\qquad a=\kappa(v_{1})+\kappa(v_{2})\}|\\
        &=\sum_{a\in \kappa(X)}|\{(v_{1},v_{2})\in V_{1}\times V_{2}\, \| \,\forall w\in\{x,y\}:(B(w,v_{1}))^{2}-a\kappa(w)\in S_{q}\text{ and }
        a=\kappa(v_{1})+\kappa(v_{2})\}|\\
        &=\sum_{a\in \kappa(X)}\sum_{v_{1}\in T_{a}}|\{v_{2}\in  V_{2}\, \| \,
        a=\kappa(v_{1})+\kappa(v_{2})\}|\\
        &=\sum_{a\in \kappa(X)}\left(\sum_{\stackrel{v_{1}\in T_{a}}{\kappa(v_{1})=a}}|\{v_{2}\in  V_{2}\, \| \,
        \kappa(v_{2})=0\}|+\sum_{\stackrel{v_{1}\in T_{a}}{\kappa(v_{1})\neq a}}|\{v_{2}\in  V_{2}\, \| \,
        a=\kappa(v_{1})+\kappa(v_{2})\}|\right)\\
        &=\sum_{a\in \kappa(X)}\left(\sum_{\stackrel{v_{1}\in T_{a}}{\kappa(v_{1})=a}}\left(\left(q^{\frac{n-1}{2}}-\eps\eps'\right)\left(q^{\frac{n-3}{2}}+\eps\eps'\right)+1\right)+\sum_{\stackrel{v_{1}\in T_{a}}{\kappa(v_{1})\neq a}}\left(q^{n-2}-\eps\eps' q^{\frac{n-3}{2}}\right)\right)\\
        &=q^{\frac{n-3}{2}}\left(q^{\frac{n-1}{2}}-\eps\eps'\right)\sum_{a\in \kappa(X)}\sum_{v_{1}\in T_{a}}1+\eps\eps'q^{\frac{n-1}{2}}\sum_{a\in \kappa(X)}\sum_{\stackrel{v_{1}\in T_{a}}{\kappa(v_{1})=a}}1\\
        &=q^{\frac{n-3}{2}}\left(q^{\frac{n-1}{2}}-\eps\eps'\right)\sum_{a\in \kappa(X)}|T_{a}|+\eps\eps'q^{\frac{n-1}{2}}\sum_{a\in \kappa(X)}|\{v_{1}\in T_{a}\, \| \,\kappa(v_{1})=a\}|\:.
    \end{align*}}
    If $\alpha\in \kappa(X)$, then $\alpha\kappa(x)\in S_{q}$ and hence there are $\frac{q-3}{2}$ elements $\nu\in\FF_q$ such that $\nu^{2}-\alpha\kappa(x)\in S_{q}$ (this is a classic result in algebra, also reflected in the parameters of the Paley graph). 
    Furthermore, $v_{1}$ is uniquely determined by $B(x,v_{1})$ and $B(y,v_{1})$ since $\{x,y\}$ is a basis of $V_{1}\cong\FF^{2}_{q}$ and $B$ is anisotropic on $V_{1}$. Hence, $|T_{\alpha}|=\frac{1}{4}(q-3)^{2}$ if $\alpha\in \kappa(X)$. For $\alpha\in \kappa(X)$, we also find that
    \begin{align*}
        \left\{v_{1}\in T_{\alpha}\, \| \,\kappa(v_{1})=\alpha\right\}&=\left\{v_{1}\in V_{1}\, \| \, \forall w\in\{x,y\}:(B(w,v_{1}))^{2}-\kappa(v_{1})\kappa(w)\in S_{q},\text{ and }\kappa(v_{1})=\alpha\right\}\\
        &=\left\{v_{1}\in V_{1}\, \| \, \vspan{X,V}\text{ and }\vspan{Y,V}\text{ are secant lines, }V=\pg(v_{1}),\text{ and }\kappa(v_{1})=\alpha\right\}\\
        &=\left\{v_{1}\in V_{1}\, \| \, v_{1}\notin\{\vspan{x},\vspan{y}\},\ \ell\text{ is a secant line, and }\kappa(v_{1})=\alpha\right\}\\
        &=\begin{cases}
        \emptyset&\eps'=-1\\
        \left\{v_{1}\in V_{1}\, \| \, v_{1}\notin\{\vspan{x},\vspan{y}\} \text{ and }\kappa(v_{1})=\alpha\right\}&\eps'=1
        \end{cases}
        \:.
    \end{align*}
    Thus, if $\eps'=1$ and $\alpha\in \kappa(X)$, we find that $\left|\left\{v_{1}\in T_{\alpha}\, \| \,\kappa(v_{1})=\alpha\right\}\right|$ equals $2\cdot \frac{q-5}{2}$ since there are $\frac{q-1}{2}-2$ vector lines where $\kappa$ takes values from $\kappa(X)$, and on each vector line where $\kappa$ takes values from $\kappa(X)$ it takes the value $\alpha$ twice. Consequently,
    \begin{align*}
        |\mathcal{S}|&=q^{\frac{n-3}{2}}\left(q^{\frac{n-1}{2}}-\eps\eps'\right)\frac{1}{4}(q-3)^{2}\left(\frac{q-1}{2}\right)+\eps\eps'q^{\frac{n-1}{2}}\left(\frac{1+\eps'}{2}\right)(q-5)\left(\frac{q-1}{2}\right)\\
        &=\frac{1}{2}(q-1)q^{\frac{n-3}{2}}\left(\frac{1}{4}\left(q^{\frac{n-1}{2}}-\eps\eps'\right)(q-3)^{2}+\frac{1}{2}\eps\left(1+\eps'\right)q(q-5)\right)\\
        &=\frac{1}{8}(q-1)q^{\frac{n-3}{2}}\left(\left(q^{\frac{n-1}{2}}-\eps\eps'\right)(q-3)^{2}+2\eps\left(1+\eps'\right)q(q-5)\right)\:.
    \end{align*}
    The result follows.
\end{proof}

\begin{lm}
 \label{lem:R2relations}
    For $(X,Y)\in R_{2}$ we have the following equalities.
    {\allowdisplaybreaks \begin{align*}
        p^{2}_{1,1}&=2q^{\frac{n-3}{2}}\left(q^{\frac{n-1}{2}}-\eps\right)\\
        p^{2}_{1,2}&=\frac{1}{2}\left(q^{\frac{n-1}{2}}-\eps\right)\left(q^{\frac{n-3}{2}}\left(q-3\right)+2\eps\right)\\
        p^{2}_{1,3}&=\frac{1}{2}q^{\frac{n-3}{2}}\left(q^{\frac{n-1}{2}}-\eps \right)\left(q-1\right)\\
        p^{2}_{2,2}&=\frac{1}{8}q^{\frac{n-3}{2}}\left(\left(q^{\frac{n-1}{2}}-\eps\right)(q-3)^{2}+4\eps q(q-5)\right)\\
        p^{2}_{2,3}&=\frac{1}{8}q^{\frac{n-3}{2}}\left(q^{\frac{n-1}{2}}-\eps\right)\left(q-1\right)\left(q-3\right)\\
        p^{2}_{3,3}&=\frac{1}{8}q^{\frac{n-3}{2}}\left(q^{\frac{n-1}{2}}-\eps\right)\left(q-1\right)^{2}\\
        p^{2}_{4,4}&=\frac{1}{8}q^{\frac{n-3}{2}}\left(\left(q^{\frac{n-1}{2}}-\eps\right)(q-1)+4\eps q\right)\left(q-1\right)\\
        p^{2}_{4,5}&=\frac{1}{8}q^{\frac{n-3}{2}}\left(q^{\frac{n-1}{2}}-\eps \right)\left(q+1\right)\left(q-1\right)\\
        p^{2}_{5,5}&=\frac{1}{8}q^{\frac{n-3}{2}}\left(q^{\frac{n-1}{2}}-\eps \right)\left(q+1\right)^{2}
    \end{align*}}
\end{lm}

\begin{proof}
 We follow the same strategy as in the proof of \Cref{lem:R1relations}.
 The intersection number $p^2_{2,2}$ was already computed in \Cref{lem:p222}.
 If we compute $p^2_{1,1}$, $p^2_{1,3}$, and $p^2_{3,3} + p^2_{5,5}$, we can apply \Cref{LmValencies}, \Cref{LmSumValency}, and \Cref{LmIntersectionNumbersEqualToZero} to find all intersection numbers.
 
 From \Cref{LmPlanesThroughLines} we know the number of planes of each type through $\ell$.
 Take a plane $\pi$ through $\ell$.
 \begin{itemize}
  \item If $\pi$ is of type $t_2$, then any anisotropic point of $\pi$ lies on a secant line through $X$ or a secant line through $Y$.
  \item If $\pi$ is a plane of type $t_{\kappa(X)}$, then it contains two tangent lines through $X$ and two tangent lines through $Y$.
  Since distinct tangent lines intersect in an anisotropic point, $\pi$ contains 4 points of $R_1(X) \cap R_1(Y)$.
  Likewise, the intersection point of any line and a passant line is an anisotropic point.
  Since $X$ and $Y$ lie on $\frac{q-1}2$ passant lines in $\pi$, this implies that $\pi$ contains $2 \frac{q-1}2$ points of $R_1(X) \cap R_3(Y)$, and $\left( \frac{q-1}2\right)^2$ points of $(R_3(X) \cap R_3(Y)) \cup (R_5(X) \cap R_5(Y))$.
  \item If $\pi$ is of type $t_{\ns_q \cdot \kappa(X)}$, $X$ and $Y$ both lie on zero tangent lines and $\frac{q+1}2$ passant lines in $\pi$.
  Therefore, $\pi$ contains no points of $R_1(X)$ and $\left( \frac{q+1}2 \right )^2$ points of $(R_3(X) \cap R_3(Y)) \cup (R_5(X) \cap R_5(Y))$.
 \end{itemize}
 This yields the following equalities.
 \begin{align*}
  p^2_{1,1} &=\frac{1}{2}\left(q^{n-2}-\eps q^{\frac{n-3}{2}}\right)\cdot4\:,\\
  p^2_{1,3}&=\frac{1}{2}\left( q^{n-2}-\eps q^{\frac{n-3}{2}} \right ) \cdot 2 \left( \frac{q-1}{2} \right )\:,\\
  p^2_{3,3}+p^2_{5,5}&=\frac{1}{2}\left(q^{n-2}-\eps q^{\frac{n-3}{2}}\right)\left(\frac{q-1}{2}\right)^{2}+\frac{1}{2}\left(q^{n-2}-\eps q^{\frac{n-3}{2}}\right)\left(\frac{q+1}{2}\right)^{2}\:.
 \end{align*}
 As stated above, the rest of the calculations follow from \Cref{LmValencies} \Cref{LmSumValency} and \Cref{LmIntersectionNumbersEqualToZero}.
 Note that $p^{2}_{0,2}=1$.
\end{proof}

\begin{lm}
 \label{lem:R3relations}
    For $(X,Y)\in R_{3}$ we have the following equalities.
    {\allowdisplaybreaks \begin{align*}
        p^{3}_{1,1}&=2q^{\frac{n-3}{2}}\left(q^{\frac{n-1}{2}}+\eps \right)\\
        p^{3}_{1,2}&=\frac{1}{2}q^{\frac{n-3}{2}}\left(q^{\frac{n-1}{2}}+\eps \right)\left(q-3\right)\\
        p^{3}_{1,3}&=\frac{1}{2}\left(q^{\frac{n-1}{2}}+\eps\right)\left(q^{\frac{n-1}{2}}-q^{\frac{n-3}{2}}-2\eps\right)\\
        p^{3}_{2,2}&=\frac{1}{8}q^{\frac{n-3}{2}}\left(q^{\frac{n-1}{2}}+\eps\right)\left(q-3\right)^{2}\\
        p^{3}_{2,3}&=\frac{1}{8}q^{\frac{n-3}{2}}\left(q^{\frac{n-1}{2}}+\eps \right)\left(q-1\right)\left(q-3\right)\\
        p^{3}_{3,3} &= \frac{1}{8}q^{\frac{n-3}{2}}\left( \left(q^{\frac{n-1}{2}}+\eps\right)(q-1)^{2}-4\eps q(q-3) \right)\\
        p^{3}_{4,4}&=\frac{1}{8}q^{\frac{n-3}{2}}\left(q^{\frac{n-1}{2}}+\eps \right)\left(q-1\right)^2\\
        p^{3}_{4,5}&=\frac{1}{8}q^{\frac{n-3}{2}}\left(q^{\frac{n-1}{2}}+\eps \right)\left(q-1\right)(q+1)\\
        p^{3}_{5,5} &= \frac{1}{8}q^{\frac{n-3}{2}}\left( \left(q^{\frac{n-1}{2}}+\eps\right)(q+1)-4\eps q \right)\left(q+1\right)
    \end{align*}}
\end{lm}

\begin{proof}
 We repeat the procedure from above.
 Recall that we computed $p^3_{2,2}$ in \Cref{lem:p222}.
 Now we compute $p^3_{1,1}$, $p^3_{1,2}$ and $p^3_{2,2} + p^3_{4,4}$.
 Take a plane $\pi$ through $\ell$.
 \begin{itemize}
  \item If $\pi$ is of type $t_0$, then any anisotropic point of $\pi$ lies on a passant line through $X$ or a passant line through $Y$.
  \item If $\pi$ is a plane of type $t_{\kappa(X)}$, then it contains two tangent lines through $X$ and two tangent lines through $Y$.
  Since distinct tangent lines intersect in an anisotropic point, $\pi$ contains 4 points of $R_1(X) \cap R_1(Y)$.
  Now take a tangent line $\ell_X$ through $X$.
  There are $\frac{q-1}2$ secant lines through $Y$, one of which intersects $\ell_X$ in a point of $\mq$.
  Hence, $\pi$ contains $2\left( \frac{q-1}2 - 1 \right)$ points of $R_1(X) \cap R_2(Y)$.
  Lastly, $X$ and $Y$ each lie on $\frac{q-1}2$ secant lines in $\pi$.
  Thus, $\pi$ contains $\left( \frac{q-1}2 \right)^2$ points lying on a secant line through $X$ and a secant line through $Y$.
  Recall that the conic $\mq \cap \pi$ contains 4 points lying on a tangent line through $X$ or $Y$, hence $q-3$ points lying on secant lines through $X$ and $Y$.
  Therefore, $\pi$ contains $\left( \frac{q-1}2 \right)^2 - (q-3)$ points of $(R_2(X) \cap R_2(Y)) \cup (R_4(X) \cap R_4(Y))$.
  \item If $\pi$ is of type $t_{\ns_q \cdot \kappa(X)}$, $X$ and $Y$ both lie on zero tangent lines and $\frac{q+1}2$ secant lines in $\pi$.
  Every point of $\mq \cap \pi$ lies on secant lines through $X$ and $Y$.
  Therefore, $\pi$ contains no points of $R_1(X)$ and $\left( \frac{q+1}2 \right )^2 - (q+1)$ points of $(R_2(X) \cap R_2(Y)) \cup (R_4(X) \cap R_4(Y))$.
 \end{itemize}
 Using \Cref{LmPlanesThroughLines}, we obtain the following equalities.
 \begin{align*}
  p^{3}_{1,1}&=\frac{1}{2}\left(q^{n-2}+\eps q^{\frac{n-3}{2}}\right)\cdot 4\:,\\
  p^{3}_{1,2}&=\frac{1}{2}\left(q^{n-2}+\eps q^{\frac{n-3}{2}}\right)\cdot 2\left(\frac{q-3}{2}\right)\:,\\
  p^{3}_{2,2}+p^{3}_{4,4}&=\frac{1}{2}\left(q^{n-2}+\eps q^{\frac{n-3}{2}}\right)\left(\left(\frac{q-1}{2}\right)^{2}-(q-3)\right)+\frac{1}{2}\left(q^{n-2}+\eps q^{\frac{n-3}{2}}\right)\left(\left(\frac{q+1}{2}\right)^{2}-(q+1)\right)\:.
 \end{align*}
 The remaining intersection numbers can be calculated using \Cref{LmValencies}, \Cref{LmSumValency}, and \Cref{LmIntersectionNumbersEqualToZero}. Recall that $p^{3}_{0,3}=1$.
\end{proof}

\begin{lm}
    For $(X,Y)\in R_{4}$ we have the following equalities.
    {\allowdisplaybreaks \begin{align*}
        p^{4}_{1,4}&=\frac{1}{2}\left(q^{\frac{n-1}{2}}-\eps\right)\left(q^{\frac{n-1}{2}}-q^{\frac{n-3}{2}}+2\eps\right)\\
        p^{4}_{1,5}&=\frac{1}{2}q^{\frac{n-3}{2}}\left(q^{\frac{n-1}{2}}-\eps \right)\left(q+1\right)\\
        p^{4}_{2,4} &= \frac{1}{8}q^{\frac{n-3}{2}}\left( \left(q^{\frac{n-1}{2}}-\eps\right)(q-1) + 4\eps q\right)\left(q-3\right)\\
        p^{4}_{2,5}&=\frac{1}{8}q^{\frac{n-3}{2}}\left(q^{\frac{n-1}{2}}-\eps \right)\left(q+1\right)\left(q-3\right)\\
        p^{4}_{3,4}&=\frac{1}{8}q^{\frac{n-3}{2}}\left(q^{\frac{n-1}{2}}-\eps \right)\left(q-1\right)^{2}\\
        p^{4}_{3,5}&=\frac{1}{8}q^{\frac{n-3}{2}}\left(q^{\frac{n-1}{2}}-\eps \right)\left(q+1\right)\left(q-1\right)
    \end{align*}}
\end{lm}

\begin{proof}
 We will start by computing $p^4_{1,5}$ and $p^4_{3,5} + p^4_{5,3}$.
 Take a plane $\pi$ through $\ell$.
 \begin{itemize}
  \item If $\pi$ is of type $t_2$, then it contains no passant lines.
  \item If $\pi$ is of type $t_{\kappa(X)}$, then it contains two tangent lines through $X$, $\frac{q-1}2$ passant lines through $X$, and $\frac{q+1}2$ passant lines through $Y$.
  Since each point of a passant line is anisotropic, $\pi$ contains $2 \frac{q+1}2$ points of $R_1(X) \cap R_5(Y)$ and $\frac{q-1}2 \cdot\frac{q+1}2$ points of $(R_3(X) \cap R_5(Y)) \cup (R_5(X) \cap R_3(Y))$.
  \item If $\pi$ is of type $t_{\ns_q \cdot \kappa(X)}$, then it contains no tangent lines through $X$, $\frac{q+1}2$ passant lines through $X$, and $\frac{q-1}2$ passant lines through $Y$.
  This again yields $\frac{q-1}2 \cdot\frac{q+1}2$ points of $(R_3(X) \cap R_5(Y)) \cup (R_5(X) \cap R_3(Y))$.
 \end{itemize}
 By \Cref{lmHulp}, $p^4_{3,5} = p^4_{5,3}$.
 Using \Cref{LmPlanesThroughLines} this implies that
 \begin{align*}
  p^{4}_{1,5} &= \frac{1}{2}\left( q^{n-2}-\eps q^{\frac{n-3}{2}} \right) \cdot2 \left( \frac{q+1}{2} \right )\:,\\
  2p^{4}_{3,5} &= \left( q^{n-2}-\eps q^{\frac{n-3}{2}} \right) \left( \frac{q-1}{2} \right) \left( \frac{q+1}{2} \right) \:.
 \end{align*}
 The remaining intersection numbers can be calculated from \Cref{LmValencies}, \Cref{LmSumValency}, and \Cref{LmIntersectionNumbersEqualToZero}.
\end{proof}

\begin{lm}
    For $(X,Y)\in R_{5}$ we have the following equalities.
    \begin{align*}
        p^{5}_{1,4}&=\frac{1}{2}q^{\frac{n-3}{2}}\left(q^{\frac{n-1}{2}}+\eps \right)\left(q-1\right)\\
        p^{5}_{1,5}&=\frac{1}{2}\left(q^{\frac{n-1}{2}}+\eps\right)\left(q^{\frac{n-1}{2}}+q^{\frac{n-3}{2}}-2\eps\right)\\
        p^{5}_{2,4}&=\frac{1}{8}q^{\frac{n-3}{2}}\left(q^{\frac{n-1}{2}}+\eps \right)\left(q-1\right)\left(q-3\right)\\
        p^{5}_{2,5}&=\frac{1}{8}q^{\frac{n-3}{2}}\left(q^{\frac{n-1}{2}}+\eps \right)(q+1)(q-3)\\
        p^{5}_{3,4}&=\frac{1}{8}q^{\frac{n-3}{2}}\left(q^{\frac{n-1}{2}}+\eps \right)\left(q-1\right)^2\\
        p^{5}_{3,5}&=\frac{1}{8}q^{\frac{n-3}{2}}\left(\left( q^{\frac{n-1}2}+\eps\right)(q+1) - 4\eps q \right)\left(q-1\right)
    \end{align*}
\end{lm}

\begin{proof}
 We begin by calculating $p^5_{1,4}$ and $p^5_{2,4} + p^5_{4,2}$.
 Take a plane $\pi$ through $\ell$.
 \begin{itemize}
  \item If $\pi$ is of type $t_0$, it contains no secant lines.
  \item If $\pi$ is of type $t_{\kappa(X)}$, then it contains two tangent lines through $X$, $\frac{q-1}2$ secant lines through $X$, and $\frac{q+1}2$ secant lines through $Y$.
  Take a tangent line $\ell_X$ through $X$.
  Then there is a unique secant line through $Y$ intersecting $\ell_X$ in a point of $\mq$, namely $\vspan{Y, \ell_X \cap \mq}$.
  Hence, $\pi$ contains $2\left( \frac{q+1}2 - 1 \right)$ points of $R_1(X) \cap R_4(Y)$.
  There are $\frac{q-1}2\cdot \frac{q+1}2$ points of $\pi$ lying on secant lines through $X$ and $Y$.
  These include all the points of the conic $\mq \cap \pi$ except for the two points lying on a tangent line through $X$.
  Thus, $\pi$ contains $\frac{q-1}2 \cdot\frac{q+1}2 - (q-1)$ points of $(R_2(X) \cap R_4(Y)) \cup (R_4(X) \cap R_2(Y))$.
  \item If $\pi$ is of type $t_{\ns_q \cdot \kappa(X)}$, then it contains no tangent lines through $X$, $\frac{q+1}2$ passant lines through $X$, and $\frac{q-1}2$ passant lines through $Y$.
  Similarly as in the previous point $\pi$ contains $\frac{q-1}2 \cdot\frac{q+1}2 - (q-1)$ points of $(R_2(X) \cap R_4(Y)) \cup (R_4(X) \cap R_2(Y))$.
 \end{itemize}
 By \Cref{lmHulp}, $p^5_{2,4} = p^5_{4,2}$.
 Using \Cref{LmPlanesThroughLines} this yields
 \begin{align*}
  p^{5}_{1,4} &= \frac{1}{2} \left( q^{n-2} + \eps q^{\frac{n-3}{2}} \right) \cdot 2 \left( \frac{q-1}{2} \right) \:, \\
  2p^{5}_{2,4} &= \frac{1}{4} \left( q^{n-2}+\eps q^{\frac{n-3}{2}} \right) (q-1) (q-3) \: .
 \end{align*}
 The remaining intersection numbers once again follow from \Cref{LmValencies}, \Cref{LmSumValency}, and \Cref{LmIntersectionNumbersEqualToZero}.
\end{proof}

\subsection{The matrices of eigenvalues and dual eigenvalues}\label{SecPQmatrix}

We will describe how to compute the eigenvalues of the association scheme in an efficient way.
We will use the notation from \Cref{SubSec:Assoc}.
In addition, the following notation will also be helpful.

\begin{df}
 Given a vector $v \in \RR^\pp$, let $v^{S_q}$ denote the vector defined by
 \[
  v^{S_q} (X) = \begin{cases}
   v(X) & \text{if } \kappa(X) = S_q, \\
   -v(X) & \text{if } \kappa(X) = \ns_q.
  \end{cases}
 \]
\end{df}

First of all, the first eigenspace $V_0$ of the scheme is spanned by the all-one vector $\one$, and $A_i \one = n_i \one$.
Since the scheme is imprimitive, respecting an equivalence relation with two classes, the vector $\one^{S_q}$ spans another 1-dimensional eigenspace of the scheme, say $V_1$.
Note that $A_i \one^{S_q}$ equals $n_i \one^{S_q}$ if $i \leq 3$ and $-n_i \one^{S_q}$ if $i \geq 4$.

Consider the intersection matrices $B_1, \dots, B_5$ and recall that the columns of $\mathbf Q$ form the (up to reordering and rescaling) unique basis that simultaneously diagonalises all intersection matrices.
Note that for $i \leq 3$,
\(
 B_i = \begin{pmatrix}
  C_i & O \\ O & D_i
 \end{pmatrix}
\)
for some $C_i \in \RR^{4\times 4}$ and $D_i \in \RR^{2\times 2}$.
Similarly, for $i \geq 4$,
\(
 B_i = \begin{pmatrix}
  O & C_i \\ D_i & O
 \end{pmatrix}
\)
for some $C_i \in \RR^{4 \times 2}$ and $D_i \in \RR^{2 \times 4}$.

The spectrum of $B_4$ is symmetric around zero.
Indeed, for $v \in \RR^4$ and $w \in \RR^2$, it holds that 
\[
B_4 \begin{pmatrix} v \\ w \end{pmatrix} = \lambda \begin{pmatrix} v \\ w \end{pmatrix} 
 \iff 
B_4 \begin{pmatrix} v \\ -w \end{pmatrix} = -\lambda \begin{pmatrix} v \\ -w \end{pmatrix}.
\]
Moreover, $\lambda$ is an eigenvalue of $B_4$ if and only if $\lambda^2$ is an eigenvalue of $B_4^2 = \begin{pmatrix}
 C_4 D_4 & O \\ O & D_4 C_4
\end{pmatrix}$, and $C_4 D_4$ has the same non-zero eigenvalues as $D_4 C_4$ with the same multiplicities.
Since $C_4$ and $D_4$ are both non-constant matrices with constant row sums, they must have rank 2.
We conclude that any basis that diagonalises $B_4$, after reordering, must be of the form
\begin{align*}
 v_0 = \begin{pmatrix} \one \\ \one \end{pmatrix}, &&
 v_1 = \begin{pmatrix} \one \\ -\one \end{pmatrix}, &&
 v_2 = \begin{pmatrix} u_2 \\ w_2 \end{pmatrix}, &&
 v_3 = \begin{pmatrix} u_2 \\ -w_2 \end{pmatrix}, &&
 v_4 = \begin{pmatrix} u_4 \\ \zero \end{pmatrix}, &&
 v_5 = \begin{pmatrix} u_5 \\ \zero \end{pmatrix}, &&
\end{align*}
where the top and bottom parts of these vectors lie in $\RR^4$ and $\RR^2$ respectively, where $v_2$ has a positive eigenvalue and $u_2, w_2 \neq \zero$, and where $u_4, u_5 \in \ker(D_4)$.

We can now compute the missing eigenvalues of $B_4$.
Let $\Tr$ denote the trace function on square matrices, i.e.\ the sum of the elements on the main diagonal.
If a matrix $A$ has eigenvalues $\lambda_i$ with respective multiplicities $m_i$, then
\begin{align} \label{Eq:Trace}
 \Tr(A^k) = \sum_i m_i \lambda_i^k,
\end{align}
for every non-negative integer $k$.
We find that
\[
 \Tr(C_4 D_4) = \Tr(D_4 C_4) = \sum_{i=0}^3 \sum_{j=4}^5 p^i_{4,j} p^j_{4,i} = n_4^2 + \left( \frac 14 q^\frac{n-3}2 (q^2-1) \right)^2.
\]
Thus, the eigenvalues of $v_2$ and $v_3$ are $\pm \frac 14 q^\frac{n-3}2 (q^2-1)$.

The eigenvalues of $B_5$ are easy to compute from the eigenvalues of $B_4$ using the fact that $A_5 = \begin{pmatrix} O & J \\ J & O \end{pmatrix} - A_4$.

We now turn our attention to the matrices $B_i$ with $i \leq 3$, and sketch how to compute their eigenvalues.
Since $v_2$ is an eigenvector of $B_i$ for some eigenvalue $\lambda$, $u_2$ and $w_2$ are $\lambda$-eigenvectors of $C_i$ and $D_i$ respectively.
In particular, every eigenvalue of $D_i$ is an eigenvalue of $C_i$, and $v_2$ and $v_3$ have the same eigenvalue with respect to $B_i$.
This eigenvalue is easy to compute as it equals $\Tr(D_i) - n_i$.
It remains to determine the eigenvalues of $C_i$ for $u_4$ and $u_5$.
We have already computed two eigenvalues of $C_i$.
The other two eigenvalues can be computed from $\Tr(C_i)$ and $\det(C_i)$ since these equal the sum and product of the eigenvalues of $C_i$ respectively.
This allows us to find the sets $\{p_i(4), p_i(5)\}$ for $i \leq 3$.
It remains to determine which of these eigenvalues in each of these sets correspond to the same eigenspace.
This can be done easily using \Cref{Res:SumEigenvalues}.

From these arguments, the matrix $\mathbf P$ can be computed.
To compute $\mathbf Q$, it suffices to compute the multiplicities $m_0, \dots, m_5$ and use \Cref{ResIdentityEigenvalueMatrices}.
We have already determined that $m_0 = m_1 = 1$.
Moreover, $v$ is a vector of the eigenspace $V_2$ if and only if $v^{S_q}$ is a vector of $V_3$.
Therefore, $m_2 = m_3$.
Furthermore, the multiplicities $m_i$ form the top row of $\mathbf Q$ and since $\mathbf Q \mathbf P = |\pp| I$, it follows that
\[
 \begin{pmatrix}
  m_0 & \dots & m_5
 \end{pmatrix} \mathbf P = \begin{pmatrix}
  |\pp| & 0 & \dots & 0
 \end{pmatrix}.
\]
This yields linear equations from which $m_2 = m_3$, $m_4$, and $m_5$ can be computed.

All these computations result in the following theorem.

\begin{table}
 \centering
\begin{minipage}{.4\textwidth}
 \begin{center}
\begin{sideways}
 $\displaystyle \mathbf P = \begin{pmatrix}
 1 & q^{n-1}-1 
   & \frac14 q^{\frac{n-1}2}\left(q^{\frac{n-1}2} + \eps\right)(q-3) 
   & \frac14 q^{\frac{n-1}2}\left(q^{\frac{n-1}2} - \eps\right)(q-1) 
   & \frac14 q^{\frac{n-1}2}\left(q^{\frac{n-1}2} + \eps\right)(q-1) 
   & \frac14 q^{\frac{n-1}2}\left(q^{\frac{n-1}2} - \eps\right)(q+1) \\[1em]
 1 & q^{n-1}-1 
   & \frac14 q^{\frac{n-1}2}\left(q^{\frac{n-1}2} + \eps\right)(q-3) 
   & \frac14 q^{\frac{n-1}2}\left(q^{\frac{n-1}2} - \eps\right)(q-1) 
   & -\frac14 q^{\frac{n-1}2}\left(q^{\frac{n-1}2} + \eps\right)(q-1) 
   & -\frac14 q^{\frac{n-1}2}\left(q^{\frac{n-1}2} - \eps\right)(q+1) \\[1em]
 1 & \eps q^{\frac{n-3}2}-1 
   & \frac14 \eps q^{\frac{n-3}2}(q+1)(q-3)
   & -\frac14 \eps q^{\frac{n-3}2}(q-1)^2
   & \frac14 \eps q^{\frac{n-3}2}(q-1)(q+1)
   & - \frac14 \eps q^{\frac{n-3}2}(q-1)(q+1) \\[1em]
 1 & \eps q^{\frac{n-3}2}-1 
   & \frac14 \eps q^{\frac{n-3}2}(q+1)(q-3)
   & -\frac14 \eps q^{\frac{n-3}2}(q-1)^2
   & -\frac14 \eps q^{\frac{n-3}2}(q-1)(q+1)
   & \frac14 \eps q^{\frac{n-3}2}(q-1)(q+1) \\[1em]
 1 & \eps q^{\frac{n-1}2} - 1 & -\eps q^{\frac{n-1}2} & 0 & 0 & 0 \\[1em]
 1 & -\eps q^{\frac{n-1}2} - 1 & 0 & \eps q^{\frac{n-1}2} & 0 & 0
 \end{pmatrix}$
\end{sideways}
 \end{center}
\end{minipage}
\begin{minipage}{.4\textwidth}
 \begin{center}
\begin{sideways}
 $\displaystyle \mathbf Q = \begin{pmatrix}
 1 & 1 
  & q^2\frac{q^{n-1}-1}{q^2-1} 
  & q^2\frac{q^{n-1}-1}{q^2-1} 
  & \frac{q-3}{2(q-1)} \left( q^{\frac{n-1}2}+\eps\right) \left(q^{\frac{n+1}2}-\eps \right) 
  & \frac{q-1}{2(q+1)} \left( q^{\frac{n-1}2}-\eps\right) \left(q^{\frac{n+1}2}-\eps \right) \\[1em]
 1 & 1 
  & q^2 \frac{\eps q^{\frac{n-3}2}-1 }{ q^2-1 }
  & q^2 \frac{\eps q^{\frac{n-3}2}-1 }{ q^2-1 }
  & \eps \frac{q-3}{2(q-1)} \left( q^{\frac{n+1}2} - \eps \right)
  & -\eps \frac{q-1}{2(q+1)} \left( q^\frac{n+1}2 - \eps \right) \\[1em]
 1 & 1 
  & \eps \frac q{q-1} \left( q^{\frac{n-1}2} - \eps \right)
  & \eps \frac q{q-1} \left( q^{\frac{n-1}2} - \eps \right)
  & - \eps \frac{2}{q-1} \left( q^\frac{n+1}2 - \eps \right)
  & 0 \\[1em]
 1 & 1 
  & -\eps \frac q{q+1} \left( q^{\frac{n-1}2} + \eps \right)
  & -\eps \frac q{q+1} \left( q^{\frac{n-1}2} + \eps \right)
  & 0
  & \eps \frac2{q+1} \left( q^\frac{n+1}2 - \eps \right) \\[1em]
 1 & -1 
  & \eps \frac q{q-1} \left( q^{\frac{n-1}2} - \eps \right)
  & - \eps\frac q{q-1} \left( q^{\frac{n-1}2} - \eps \right) 
  & 0
  & 0 \\[1em]
 1 & -1 
  & - \eps \frac q{q+1} \left( q^{\frac{n-1}2} + \eps \right)
  & \eps \frac q{q+1} \left( q^{\frac{n-1}2} + \eps \right)
  & 0
  & 0
 \end{pmatrix}$
\end{sideways}
 \end{center}
\end{minipage}
 \caption{The matrices of eigenvalues and dual eigenvalues of the association scheme.}
 \label{Tab:Matrices}
\end{table}

\begin{thm}
 Consider the quadric $\mq^\eps(n,q)$ with $\eps = \pm 1$, and $q$ odd.
 The relations $R_0, \dots, R_5$ defined by \Cref{tab:rel} constitute an association scheme on the anisotropic points of $\mq^\eps(n,q)$.
 The matrices of eigenvalues and dual eigenvalues of this scheme are presented in \Cref{Tab:Matrices}.
\end{thm}

\begin{rmk}
 This association scheme is imprimitive.
 It has two isomorphic primitive subschemes, obtained by restricting to anisotropic points of one quadratic type.
 Such a subscheme has 3 classes if $q>3$ and 2 classes if $q=3$.
 The matrix of eigenvalues of this subscheme can be deduced from the matrix $\mathbf P$ in \Cref{Tab:Matrices} by removing the last two columns and removing rows 1 and 3.
\end{rmk}

Of special interest is the case where $n=3$.
By applying the polarity, we can interpret this association scheme as being defined on the non-degenerate plane sections of $\mq^\eps(3,q)$.
This gives us an association scheme on the circles of the miquelian Möbius or Minkowski plane of order $q$, depending on whether $\eps = -$ or $\eps = +$ respectively.
The investigation of known circle geometries was actually the initial motivation for this paper.
In particular, we can use the alternative representation of the miquelian Möbius and Minkowski planes to give alternative interpretations of this association scheme.

If $\eps=-1$, then we can interpret this scheme as being defined on the Baer sublines of $\pg(1,q^2)$.
Take two distinct Baer sublines.
If they lie in the same $\PSL(2,q^2)$ orbit, then they are in relation $R_1$, $R_2$, or $R_3$ depending on whether they intersect in $1$, $0$, or $2$ points respectively.
It they lie in different $\PSL(2,q^2)$ orbits, they are in relation $R_4$ or $R_5$ depending on whether they intersect in $0$ or $2$ points respectively.

If $\eps=+1$, then we can interpret this scheme as being defined on the elements of $\PGL(2,q)$.
The relation containing two elements $f,g$ depends on whether they lie in the same coset of $\PSL(2,q)$ and on how many fixed points $f \circ g^{-1}$ has.
This scheme was already described, and its matrices of eigenvalues and dual eigenvalues determined, by the first author \cite[\S 3]{adriaensen2022stability}.

\subsection{Combinatorial description of the eigenspaces.}

We already know that $V_0 = \vspan \one$ and $V_1 = \vspan{\one^{S_q}}$.
We can say more about the eigenspaces.
For a subspace $\pi$ of $\pg(n,q)$, let $\chi_\pi$ denote the characteristic vector of the anisotropic points contained in $\pi$.
We will look for cliques and cocliques reaching equality in Delsarte's or (weighted) Hoffman's ratio bound (Results \ref{Res:Delsarte} and \ref{Res:Hoffman}) and use \Cref{Lm:SpanningOrbit} to describe certain eigenspaces of the association scheme.
In this regard, it is important to observe that the orthogonal group $\PGO^\eps(n+1,q)$ associated to $\mq^\eps(n,q)$ acts transitively on all subspaces of $\pg(n,q)$ that intersect $\mq^\eps(n,q)$ in isomorphic quadrics, see \cite[Theorem 1.49]{hirschfeldthas}.
In particular, $\PGO^\eps(n,q)$ acts transitively on the passant lines, the tangent lines, and the secant lines of $\mq^\eps(n,q)$.
It follows that the action of $\PGO^\eps(n+1,q)$ on the set of anisotropic points $\pp$ yields a group of automorphisms that satisfies the condition of \Cref{Lm:SpanningOrbit}.

\begin{prop}
 \label{Prop:R1Delsarte}
 Let $\ms_1$ denote the set of $\frac{n-1}2$-spaces that intersect $\mq^\eps(n,q)$ exactly in a totally isotropic $\frac{n-3}2$-space.
 For every $\pi \in \ms_1$, $\pi \cap \pp$ is a Delsarte clique for the relation $R_1$.
 Moreover, $\vspann{\chi_\pi}{\pi \in \ms_1} = V_\frac{9+\eps}2^\perp$.
\end{prop}

\begin{proof}
 If we apply Delsarte's ratio bound to relation $R_1$, we obtain the bound 
 \[
  \frac{q^{n-1}-1}{q^{\frac{n-1}2}+1} + 1 = q^\frac{n-1}2.
 \]
 The eigenspace of eigenvalue $-(q^{\frac{n-1}2}+1)$ is $V_{\frac{9+\eps}2}$.
 If $\pi \in \ms_1$, let $\rho$ denote $\pi \cap \mq^\eps(n,q)$.
 Then every line $\ell$ contained in $\pi$ is either contained in $\rho$ and therefore totally isotropic, or intersects $\rho$ in a unique point and therefore tangent.
 It follows that any two anisotropic points of $\pi$ span a tangent line, and that $\pi$ contains $\theta_{\frac{n-1}2}(q) - \theta_{\frac{n-3}2}(q) = q^\frac{n-1}2$ anisotropic points.

 Now consider the matrix $M \in \RR^{\pp \times \ms_1}$ with the vectors $\chi_\pi$, $\pi \in \ms_1$, as columns.
 By considering the automorphisms, we see that the number of spaces $\pi \in \ms_1$ through two points $X,Y \in \pp$ only depends on which relation $R_i$ contains $(X,Y)$.
 Moreover, this number equals 0 if $i > 1$, hence $M M^\top$ is a linear combination of $I$ and $A_1$.
 Therefore, the column space of $M M^\top$ is spanned by all or all but one of the eigenspaces of $A_1$.
 Note that the column spaces of $M$ and $M M^\top$ coincide and equal $\vspann{\chi_\pi}{\pi \in \ms_1}$.
 By equality in Delsarte's ratio bound, the latter space is orthogonal to $V_{\frac{9+\eps}2}$.
 Hence, the only possibility is that $\vspann{\chi_\pi}{\pi \in \ms_1} = V_{\frac{9+\eps}2}^\perp$.
\end{proof}

\begin{prop}
 \label{Prop:RDisjointHoffman}
 Let $\ms_2$ denote the set of totally isotropic $\frac{n-3}2$-spaces.
 For every $\pi \in \ms_2$, $\pi^\perp \cap \pp$ is a weighted Hoffman coclique for the relation $R_\frac{5+\eps}2 \cup R_\frac{9+\eps}2$ of size $q^\frac{n-1}2(q-\eps)$.
 Moreover, $\vspann{\chi_{\pi^\perp}}{\pi \in \ms_2} = V_0 \oplus V_2$ and $\vspann{\chi_{\pi^\perp}^{S_q}}{\pi \in \ms_2} = V_1 \oplus V_3$.
\end{prop}

\begin{proof}
 If $\pi \in \ms_2$, then $\pi^\perp$ intersects $\mq^\eps(n,q)$ in a cone $\pi \mq^\eps(1,q)$.
 This implies that $\pi^\perp$ contains no lines intersecting $\mq^\eps(n,q)$ in exactly $1-\eps$ points, and hence that the anisotropic points of $\pi^\perp$ are a coclique for relations $R_{\frac{5+\eps}2}$ and $R_{\frac{9+\eps}{2}}$.
 The number of anisotropic points in $\pi^\perp$ equals
 \[
  \theta_{\frac{n+1}2}(q) - |\pi \mq^\eps(1,q)|
  = \theta_{\frac{n+1}2}(q) - \left( \theta_{\frac{n-3}2}(q) + (1+\eps) q^\frac{n-1}2 \right)
  = (q-\eps) q^\frac{n-1}2.
 \]
 We will check that this matches a weighted Hoffman's ratio bound.
 Note that any linear combination of $A_{\frac{5+\eps}2}$ and $A_{\frac{9+\eps}2}$ satisfies the conditions of \Cref{Res:Hoffman} with respect to the relevant graph.
 The eigenvalues of this linear combination can be easily obtained from the matrix $\mathbf P$ from \Cref{Tab:Matrices}.

 When making a linear combination of $A_{\frac{5+\eps}2}$ and $A_{\frac{9+\eps}2}$, we want it to have its smallest eigenvalue on $V_2$.
 In particular, we need to avoid that the smallest eigenvalue corresponds to $V_1$.
 Therefore, we take the linear combination $A = (q+\eps) A_{\frac{5+\eps}2} + (q-2+\eps) A_{\frac{9+\eps}2}$.
 Then $\one$ is an eigenvector of $A$ with eigenvalue
 \[
  (q+\eps) \mathbf P\left(0,\frac{5+\eps}2\right) + (q-2+\eps) \mathbf P\left(0,\frac{9+\eps}2\right)
  = \frac 12 q^\frac{n-1}2 \left( q^\frac{n-1}2 - 1 \right) (q+\eps) (q-2+\eps)
 \]
 and $A$ has its smallest eigenvalue
 \[
  (q+\eps) \mathbf P\left(2,\frac{5+\eps}2\right) + (q-2+\eps) \mathbf P\left(2,\frac{9+\eps}2\right)
  = - \frac 12 q^\frac{n-3}2 (q-\eps) (q+\eps) (q-2+\eps)
 \]
 on eigenspace $V_2$.
 Plugging this into \Cref{Res:Hoffman} yields a bound $(q-\eps) q^\frac{n-1}2$ on cocliques for $R_{\frac{5+\eps}2} \cup R_{\frac{9+\eps}2}$, which is achieved by the set of anisotropic points in a space $\pi^\perp$ with $\pi \in \ms$.
 Then $V_0 \oplus V_2 = \vspann{\chi_{\pi^\perp}}{\pi \in \ms_2}$ by \Cref{Lm:SpanningOrbit}.
 Since the bipartite graph $(\pp,R_4)$ has eigenvalues with opposite sign on $V_0$ and $V_1$, and on $V_2$ and $V_3$, we know that $V_1 = \sett{v^{S_q}}{v \in V_0}$ and $V_3 = \sett{v^{S_q}}{v \in V_2}$.
 Thus, $V_1 \oplus V_3 = \vspann{\chi_{\pi^\perp}^{S_q}}{\pi \in \ms_2}$.
\end{proof}

\begin{rmk}
 Consider again the case $n=3$.
 As mentioned before, we can interpret this association scheme as being defined on the circles of the miquelian Möbius or Minkwoski plane of order $q$.
 In this interpretation, two circles are in relation $R_\frac{5+\eps}2 \cup R_\frac{9+\eps}2$ if and only if they are disjoint.
 Cocliques for this relation are called \emph{intersecting families}.
 Large intersecting families in the known circle geometries have been classified by the first author \cite[Theorem 1.3]{adriaensen2022stability}.
\end{rmk}

Lastly, we study Hoffman cocliques of $R_1$ for $\mq^+(3,q)$.
These are equivalent to other interesting geometric objects.
The corresponding graph has two isomorphic connected components, corresponding to the two quadratic types.
So, we will focus on one of these components, i.e.\ take the points of one quadratic type.
In that case, Hoffman cocliques have size 
\[
 \frac{\frac 12 q(q^2-1)}{\frac{q^2-1}{q+1}+1} = \frac{q^2-1}{2}.
\]

A set of points $\mo$ in a polar space is called an \emph{ovoid} if it intersects every generator in exactly one point.
If $\mo$ intersects every generator in at most one point, we call it a \emph{partial ovoid}.
Partial ovoids of $\mq(4,q)$ have been investigated in several papers.
We review the most important properties and refer the interested reader to \cite{CoolsaetDeBeuleSiciliano}.
Note that the generators of $\mq(4,q)$ are lines.
Let $\kappa$ and $\perp$ denote the quadratic form and polarity associated to $\mq(4,q)$, respectively.

Ovoids of $\mq(4,q)$ contain $q^2+1$ points.
Partial ovoids of $\mq(4,q)$ which are not contained in any ovoid, contain at most $q^2-1$ points.
Such a partial ovoid whose size meets the upper bound will be called a \emph{maximum partial ovoid}.
If $\mo$ is a maximum partial ovoid of $\mq(4,q)$, then there exists a hyperplane $\pi$ intersecting $\mq(4,q)$ in a $\mq^+(3,q)$ such that the generators of $\mq(4,q)$ that miss $\mo$ are exactly the generators contained in $\pi$.
Let $P$ denote the point $\pi^\perp$.
We call the maximum partial ovoid $\mo$ \emph{antipodal} if every secant line through $P$ intersects $\mo$ in either 0 or 2 points.

A line $\ell$ through $P$ is tangent to $\mq(4,q)$ if and only if it intersects $\pi$ in a point of the quadric $\mq^+(3,q)$.
If $\ell$ is not tangent, then $\ell$ is passant or secant depending on the quadratic type of $\ell \cap \pi$.
Since we can scale $\kappa$ if necessary, we may suppose without loss of generality that the secant lines through $P$ intersect $\pi$ in a point of square type.

Now suppose that $\mo$ is an antipodal maximum partial ovoid.
Let $C$ denote its projection from $P$ onto $\pi$, i.e.\ $C$ equals the set $\sett{ \vspan{P,R} \cap \pi }{R \in \mo}$.
Since $\mo$ is antipodal, $|C| = \frac{|\mo|}2 = \frac{q^2-1}{2}$.
Note also that $C$ consists of anisotropic points of $\pi$ of square type.
Take a tangent line $\ell$ in $\pi$ whose anisotropic points are of square type.
Suppose that $\ell$ intersects $\mq(4,q)$ in the point $Q$.
Then $\vspan{P,Q}$ is a tangent line, and all other lines through $P$ in $\vspan{P,\ell}$ are secant lines.
It follows that $\vspan{P,\ell}$ meets $\mq(4,q)$ in the union of two lines through $Q$.
Both of these lines must contain a point of $\mo$, say $R_1$ and $R_2$.
Since $\mo$ is antipodal, $P$, $R_1$, and $R_2$ must be collinear, which means that $P$ projects $R_1$ and $R_2$ onto the same point of $C$.
Therefore, $\ell$ contains a unique point of $C$.
This implies that $C$ is indeed a Hoffman coclique for the relation $R_1$ on the anisotropic points of $\mq^+(3,q)$ of square type.

This process is reversible.
Suppose that $C$ is a Hoffman coclique for the relation $R_1$ on the anisotropic points of $\mq^+(3,q)$ of square type.
For every point $R \in C$, the line $\vspan{P,R}$ is secant to $\mq(4,q)$.
Let $\mo$ denote the set of points in which these lines $\vspan{P,R}$ intersect $\mq(4,q)$.
Then $\mo$ is an antipodal maximum partial ovoid of $\mq(4,q)$.
Indeed, it readily follows that $\mo$ contains $q^2-1$ points and is antipodal.
Moreover, a generator $\ell$ of $\mq(4,q)$ cannot contain 2 points of $\mo$.
This is clear if $\ell$ is contained in $\pi$, so suppose that $\ell$ isn't.
Then the plane $\vspan{P,\ell}$ intersects $\pi$ in a tangent line.
This follows from the fact that $\ell$ contains a unique point $Q$ of $\pi = P^\perp$, hence a unique tangent line $\vspan{P,Q}$ through $P$, which implies that $Q$ is the only isotropic point of $\vspan{P,\ell} \cap \pi$.
Then $\vspan{P,\ell} \cap \pi$ cannot contain multiple points of $C$, which implies that $\ell$ cannot contain multiple points of $\mo$.

We conclude that Hoffman cocliques of $R_1$ are equivalent to antipodal maximum partial ovoids of $\mq(4,q)$.
As observed in \cite{CoolsaetDeBeuleSiciliano}, a maximum partial ovoid of $\mq(4,q)$ is equivalent to a sharply transitive subset of $\SL(2,q)$.
Moreover, this maximum partial ovoid is antipodal if and only if the sharply transitive subset is closed under multiplication with $-1$.
Hence, we have established the following equivalence.

\begin{prop}
 Let $q$ be an odd prime power.
 The following objects are equivalent.
 \begin{enumerate}
  \item A Hoffman coclique of the $R_1$ relation on the points of $\mq^+(3,q)$ of one quadratic type.
  \item An antipodal maximum partial ovoid of $\mq(4,q)$.
  \item A sharply transitive subset of $\SL(2,q)$ that is closed under multiplication with $-1$.
 \end{enumerate}
\end{prop}

A maximum partial ovoid of $\mq(4,q)$ can only exist if $q$ is a prime \cite{DeBeuleGacs}.
Examples are known for $q = 3,5,7,11$, all of which are antipodal and correspond to sharply transitive subgroups of $\SL(2,q)$.
Coolsaet, De Beule, and Siciliano \cite{CoolsaetDeBeuleSiciliano} conjecture that a sharply transitive subset of $\SL(2,q)$ can only arise from a sharply transitive subgroup, which would imply that there are no further examples of maximum partial ovoids of $\mq(4,q)$, antipodal or otherwise.

\section{Orthogonality graphs}
 \label{Sec:Ortho}

In this section we determine the spectrum of the adjacency matrices of some orthogonality graphs.
These graphs are constructed as follows.
Take a polarity $\perp$ of $\pg(n,q)$.
Take a subset $\pp$ of the points of $\pg(n,q)$.
Define a graph $G$ whose vertices are $\pp$ and where $X$ is adjacent to $Y$ if and only if $X \perp Y$.
We can ensure that the graph is loopless, by choosing only anisotropic points as vertices.
Typically, $\pp$ consists either of all anisotropic points or in case $\perp$ defines a quadric and $q$ is odd, of the anisotropic points of one quadratic type.

Orthogonality graphs gained interest in the study of Ramsey-type problems, more specifically they are ``dense'', ``clique-free'', and ``pseudorandom''.
Pseudorandom means that they behave in some way like random graphs.
For a regular graph, this is certainly the case if the second largest eigenvalue in absolute value of the adjacency matrix of $G$ is close to the square root of the degree.
Clique-free means that the clique number of the graph is small.
Dense means that the graph contains a lot of edges, or in other words that it has high degree.

The graphs were introduced in this context by Alon and Krivelevich \cite{AlonKrivelevich}.
They considered the case where $q$ is even, $\perp$ is a \emph{pseudo-polarity} (which means that the absolute points of $\perp$ form a hyperplane), and $\pp$ consists of all points of $\pg(n,q)$ (which forces the graph to contain loops).
Bishnoi, Ihringer, and Pepe \cite{BishnoiIhringerPepe} studied the case where $q$ is odd, $\perp$ gives rise to a non-degenerate quadric, and $\pp$ consists of the anisotropic points of one quadratic type.
From the viewpoint of dense clique-free pseudorandom graphs, this is a meaningful improvement to the original construction by Alon and Krivelevich.
The interested reader is invited to consult \cite{BishnoiIhringerPepe} for more details.

Determining the eigenvalues of the orthogonality graphs is easy when $\pp$ consists of all points of $\pg(n,q)$, as noted in \cite{AlonKrivelevich}.
When we restrict $\pp$ to a smaller set of points, we obtain a subgraph whose eigenvalues interlace the eigenvalues of the bigger graph.
This generally yields a good upper bound on the second largest eigenvalue in absolute value of the subgraph.
However, using the eigenvalues of the association schemes we encountered in this paper, we can compute the eigenvalues of some orthogonality graphs on anisotropic points exactly.

\begin{thm} \label{Thm:EigOrtho}
 The eigenvalues of orthogonality graphs on anisotropic points are listed below.
 \begin{enumerate}
  \item \label{Thm:EigOrtho:Herm} Hermitian variety $\mh(n,q^2)$, $\eps = (-1)^n$
  
  \begin{tabular}{c||c|c|c|c}
     Eig.\ & $q^{n-1} \frac{q^n-\eps}{q+1} $ & $\eps q^{n-1}$ & $-\eps q^{n-1}$ & $-\eps q^{n-2}$ \\ \hline
     Mult.\ & $1$ & $\frac q {2(q+1)^2} \left( q^{n+1}+\eps \right) \left(q^n - \eps \right)$ & $\frac {q-2} {2(q^2-1)} \left( q^{n+1}+\eps \right) \left(q^n - \eps \right)$ & $q^3 \frac{q^n-\eps}{q^2-1} \frac{q^{n-1}+\eps}{q+1}$
  \end{tabular}

  \item \label{Thm:EigOrtho:HypEll} Elliptic or hyperbolic quadric $\mq^\eps(n,q)$, $q$ odd

  \begin{tabular}{c||c|c|c|c}
     Eig.\ & $q^{n-1}$ & $\eps q^\frac{n-1}2$ & $- \eps q^\frac{n-1}2$ & $\eps q^\frac{n-3}2$ \\ \hline
     Mult.\ & $1$ & $\frac{q^{n+1} - 2q^n + 1}{2(q-1)} - \eps q^\frac{n-1}2$
      & $\frac{q^{n+1}-1}{2(q+1)}$ 
      & $q^2 \frac{q^{n-1}-1}{q^2-1}$
  \end{tabular}

  \item \label{Thm:EigOrtho:Par} Parabolic quadric $\mq(n,q)$, $q$ odd

  \begin{tabular}{c||c|c|c|c|c}
     Eig.  & $q^{n-1}$ & $\pm q^\frac{n-1}2$ & $q^{\frac n2 -1}$ & $-q^{\frac n2 - 1}$ & $0$\\ \hline
     Mult.  & $1$ 
     & $\frac 12 \left( q^n - q \frac{q^{n-1}-1}{q-1} \right) - 1$
     & $\frac 12 \left( q \frac{q^{n-1}-1}{q-1} - q^\frac n2 \right)$
     & $\frac 12 \left( q \frac{q^{n-1}-1}{q-1} + q^\frac n2 \right)$
     & $1$
  \end{tabular}
 \end{enumerate}
\end{thm}

We note that when $\perp$ gives rise to a quadric in $\pg(n,q)$, $q$ odd, (respectively a Hermitian variety in $\pg(n,q^2)$) and we take $\pp$ to be the set of anisotropic points of one quadratic type (respectively all anisotropic points), the eigenvalues of the orthogonality graph were described by Bannai, Hao, and Song \cite{bannaihaosong} (and also the same authors and Wei \cite{bannaisonghaowei}).
In the Hermitian case, we obtain \Cref{Prop:FissionHerm} as a corollary from our computation of the eigenvalues.

\bigskip

The general strategy to compute the eigenvalues is as follows.
Let $G$ be the orthogonality graph of $\perp$ on the anisotropic points $\pp$, and let $A$ be its adjacency matrix.
Given two points $X,Y \in \pp$, $A^2(X,Y)$ equals the number of common neighbours of $X$ and $Y$ in $G$, which equals $X^\perp \cap Y^\perp \cap \pp$.
In all of the above cases, this only depends on how $\vspan{X,Y}$ intersects the quadric or Hermitian variety defined by $\perp$.
Therefore, $A^2$ lives in the Bose-Mesner algebra of one of the previously encountered association schemes.
Hence, we can determine the eigenvalues of $A^2$ and their multiplicities.
For every eigenvalue $\lambda^2$ of $A^2$ with multiplicity $m$, it only remains to determine the multiplicities of $\lambda$ and $-\lambda$ as eigenvalues of $A$.
These latter multiplicities of course sum to $m$.

We can obtain information about the spectrum of $A$ by considering $\Tr(A^k)$ for some integers $k$, cf.\ \Cref{Eq:Trace}.
This gives a system of linear equations in $m_i$.
Moreover, $\Tr(A^k)$ gives the number of ordered closed walks of length $k$ in $G$.
In particular, $\Tr(A^0)$ equals the number of vertices, $\Tr(A)$ equals zero, and $\Tr(A^3)$ equals the number of ordered triangles.
In all cases we encounter, these equations suffice to determine the spectrum of $A$.

We now prove the three cases of \Cref{Thm:EigOrtho} in three separate subsections.

\subsection{The Hermitian case}
 \label{SecHerm}

Consider the case where $\perp$ determines a Hermitian variety $\mh(n,q^2)$ with $n \geq 2$.
Recall the relations $R_0, R_1, R_2 = R_{2\perp} \cup R_{2\not\perp}$ defined in \Cref{Sec:NU}.
Then the orthogonality graph on the anisotropic points of $\mh(n,q^2)$ is the graph corresponding to relation $R_{2\perp}$.
Let $A_i$ denote the adjacency matrix corresponding to $A_i$.
Using \Cref{DfAssocMatrices}, it suffices to prove that $\vspan{A_0=I,A_1,A_{2\perp},A_{2\not\perp}}$ is closed under matrix multiplication.
We do this by proving that all these matrices lie in the algebra generated by $A_{2\perp}$, and that this algebra is 4-dimensional.

Let $\eps$ denote $(-1)^n$.
Then
\begin{align}
 \label{EqOrthoHerm}
 A_{2\perp}^2 = q^{n-1} \frac{q^n-\eps}{q+1} I 
 + q^{n-1} \frac{q^{n-2}-\eps}{q+1} A_1
 + q^{n-2} \frac{q^{n-1}+\eps}{q+1} A_2.
\end{align}
Note that $A_2 = A_{2\perp} + A_{2\not\perp} = J - I - A_1$.
Therefore, $A_1$ is a linear combination of $A_{2\perp}^2$, $I$, and $J$.
The matrices $I$ and $J$ are inside the algebra spanned by $A_{2\perp}$.
This is obvious for $I$.
For $J$, it follows from the fact that the orthogonality graph is regular and connected, hence $\vspan \one$ is an eigenspace of $A_{2\perp}$.
The orthogonal projection onto this eigenspace lives in the algebra spanned by $A_{2\perp}$, and is a multiple of $J$.
It readily follows that $A_0=I, A_1, A_{2\perp}, A_{2\not\perp}$ all lie inside the algebra spanned by $A_{2\perp}$.
Since $A_{2\perp}$ is symmetric, the dimension of this algebra as vector subspace equals the number of distinct eigenvalues of $A_{2\perp}$.

We will now derive the spectrum of $A_{2\perp}$ from the spectrum of $A_1$, which can be read from the matrix $\mathbf P$ in \Cref{Sec:NU}.
Plugging this into \Cref{EqOrthoHerm}, yields that $A_{2\perp}^2$ has eigenvalues $q^{2(n-1)}$ and $q^{2(n-2)}$ on the orthogonal complement of $\vspan \one$.
Therefore the possible eigenvalues of $A_{2\perp}$ are
\begin{align*}
 \lambda_0 = q^{n-1} \frac{q^n-\eps}{q+1}, &&
 \lambda_1 = \eps q^{n-1}, &&
 \lambda_2 = - \eps q^{n-1}, &&
 \lambda_3 = \eps q^{n-2}, &&
 \lambda_4 = -\eps q^{n-2}.
\end{align*}
Let $m_i$ denote the multiplicity of $\lambda_i$.
First of all, we know that $m_0 = 1$.
From the top row of the matrix $\mathbf Q$ in \Cref{Sec:NU}, we know that
\begin{align*}
    m_1 + m_2 = \frac{q^2-q-1}{(q+1)(q^2-1)}\left( q^{n+1} + \eps \right) \left( q^n - \eps \right), &&
    m_3 + m_4 = \frac{q^3}{(q+1)(q^2-1)} \left( q^{n-1} + \eps \right)\left( q^n - \eps \right).
\end{align*}
We count the number of ordered triangles in the graph.
By \Cref{Res:HermitianAnisotropicPoints}, there are $q^n \frac{q^{n+1}+\eps}{q+1}$ choices for a point $X \in \pp$, $q^{n-1} \frac{q^n-\eps}{q+1}$ choices for a point $Y \in \pp \cap X^\perp$, and $q^{n-2} \frac{q^{n-1}+\eps}{q+1}$ choices for a point $Z \in X^\perp \cap Y^\perp \cap \pp$.
We may conclude that
\begin{align*}
 \sum_{i=0}^4 m_i \lambda_i = 0, &&
 \sum_{i=0}^4 m_i \lambda_i^3 = \left( q^n \frac{q^{n+1}+\eps}{q+1} \right) \left( q^{n-1} \frac{q^n-\eps}{q+1} \right) \left( q^{n-2} \frac{q^{n-1}+\eps}{q+1} \right).
\end{align*}
Using $m_0=1$, this linear system has a unique solution, namely
\begin{align*}
 m_1 &= \frac q {2(q+1)^2} \left( q^{n+1}+\eps \right) \left(q^n - \eps \right) &
 m_2 &= \frac {q-2} {2(q^2-1)} \left( q^{n+1}+\eps \right) \left(q^n - \eps \right) \\
 m_3 &= 0 &
 m_4 &= q^3 \frac{q^n-\eps}{q^2-1} \frac{q^{n-1}+\eps}{q+1} 
\end{align*}
Since $m_3 = 0$, $A_{2\perp}$ indeed has 4 distinct eigenvalues, which proves \Cref{Prop:FissionHerm}.
We have also proven \Cref{Thm:EigOrtho} (\ref{Thm:EigOrtho:Herm}).

\begin{rmk}
 Recall \Cref{Rmk:HermOrtho} (\ref{Rmk:HermOrtho:Q=2}).
 In case $q=2$, we see that $m_2 = 0$, and if in addition $n=2$, then $\lambda_0 = \lambda_1$.
\end{rmk}

\subsection{The elliptic and hyperbolic case}

Consider the case where $\perp$ determines the quadric $\mq^\eps(n,q)$, with $n$ and $q$ odd, and $\eps = \pm 1$.
Let $A$ denote the adjacency matrix of the orthogonality graph on the anisotropic points of $\mq^\eps(n,q)$.
Let $A_0, \dots, A_5$ denote the adjacency matrices of the association scheme from \Cref{Sec:NewScheme}.
Then
\begin{align*}
  A^2 = & \left( \theta_{n-1}(q) - \left| \mq(n-1,q) \right| \right) I
  + \left( \theta_{n-2}(q) - \left| \Pi_0 \mq(n-3,q) \right| \right) A_1 \\
  & + \left( \theta_{n-2}(q) - \left| \mq^\eps(n-2,q) \right| \right) (A_2+A_4)
  + \left( \theta_{n-2}(q) - \left| \mq^{-\eps}(n-2,q) \right| \right) (A_3+A_5) \\
  = & q^{n-1} I + q^{n-2} A_1 + q^\frac{n-3}2 \left( q^\frac{n-1}2 - \eps \right) (A_2+A_4) + q^\frac{n-3}2 \left( q^\frac{n-1}2 + \eps \right) (A_3+A_5) \\
  = & q^\frac{n-3}2 \left( q^\frac{n-1}2 J + q^\frac{n-1}2(q-1) I + \eps( A_3 + A_5 - A_2 - A_4) \right).
\end{align*}
Let $V_0, \dots, V_5$ denote the eigenspaces of the Bose-Mesner algebra spanned by $A_0, \dots, A_5$.
Using \Cref{Tab:Matrices}, we see that $A^2$ takes the eigenvalues $q^{2(n-1)}$ on $V_0$, $q^{n-3}$ on $V_2$, and $q^{n-1}$ on all the other eigenspaces.
Since $G$ is $q^{n-1}$-regular, $A$ has eigenvalue $q^{n-1}$ on $V_0$.
Define
\begin{align*}
 \lambda_0 = q^{n-1}, &&
 \lambda_1 = \eps q^\frac{n-1}2, &&
 \lambda_2 = -\eps q^\frac{n-1}2, &&
 \lambda_3 = \eps q^\frac{n-3}2, &&
 \lambda_4 = -\eps q^\frac{n-3}2.
\end{align*}
Let $m_i$ denote the multiplicity of $\lambda_i$ as eigenvalue of $A$.
We already know that $m_0=1$.
Moreover,
\begin{align*}
  m_1 + m_2 = q^\frac{n-1}2 \left( q^\frac{n+1}2 - \eps \right) - q^2 \frac{q^{n-1} - 1}{q^2-1} - 1, &&
  m_3 + m_4 = q^2 \frac{q^{n-1} - 1}{q^2-1}.
\end{align*}
Applying \Cref{Eq:Trace} with $k=1$ tells us that
\[
 q^{n-1} + \eps (m_1 - m_2) q^\frac{n-1}2 + \eps (m_3-m_4) q^\frac{n-3}2 = 0.
\]
Finally, we count the number of ordered triangles $(X,Y,Z)$ in the orthogonality graph.
There are $q^\frac{n-1}2 \left( q^\frac{n+1}2 - \eps \right)$ ways to choose $X$.
 Then we need to choose a anisotropic point $Y \in X^\perp$.
 Note that a point $Y \in X^\perp$ is anisotropic if and only if $\vspan{X,Y}$ is not a tangent line.
 Hence, take an non-tangent line $\ell$ through $X$, and let $Y$ denote $\ell \cap X^\perp$.
 From the proof of \Cref{LmValencies}, we know that there are $\frac 12 \left( q^{n-1}+\eps q^\frac{n-1}2 \right)$ secant lines through $X$, and $\frac 12 \left( q^{n-1}-\eps q^\frac{n-1}2 \right)$ passant lines through $X$.
 In the former case, $\ell^\perp$ intersects $\mq$ in a quadric isomorphic to $\mq^\eps(n-2,q)$ and in the latter case a quadric isomorphic to $\mq^{-\eps}(n-2,q)$.
 Therefore,
 \begin{align*}
  \sum_i \lambda_i^3 m_i = \Tr(A^3) &=  q^\frac{n-1}2 \left( q^\frac{n+1}2 - \eps \right) \left( \frac 12 \left( q^{n-1}+\eps q^\frac{n-1}2 \right) \left( \theta_{n-2}(q) - | \mq^\eps(n-2,q) | \right)\right. \\
  & \qquad+ \left.\frac 12 \left( q^{n-1}-\eps q^\frac{n-1}2 \right) \left( \theta_{n-2}(q) - | \mq^{-\eps}(n-2,q) | \right) \right) \\
  &=  q^\frac{3n-5}2 \left( q^\frac{n+1}2 - \eps \right) \left( q^{n-1} - 1 \right).
 \end{align*}

It follows that
\begin{align*}
 m_0 = 1, &&
 m_1 = \frac{q^{n+1} - 2q^n + 1}{2(q-1)} - \eps q^\frac{n-1}2, &&
 m_2 = \frac{q^{n+1}-1}{2(q+1)}, &&
 m_3 = q^2 \frac{q^{n-1}-1}{q^2-1}, &&
 m_4 = 0.
\end{align*}
This proves \Cref{Thm:EigOrtho} (\ref{Thm:EigOrtho:HypEll}).

\subsection{The parabolic case}

Consider the parabolic quadric $\mq = \mq(n,q)$, with $n$ even and $q$ odd.
As usual, let $\perp$ denote the related polarity.
Let $\pp$ denote the set of all anisotropic points.
We can partition $\pp$ into the sets
\[
 \pp^\eps = \sett{ X \in \pp }{ X^\perp \cap \mq \cong \mq^\eps(n-1,q) }
\]
for $\eps = \pm 1$.
Then $|\pp^\eps| = \frac 12 q^\frac n2 \left( q^\frac n2 + \eps \right)$.
Consider the graph $G$ defined on $\pp$, where adjacency is given by being orthogonal.
We will determine the eigenvalues of the adjacency matrix $A$ of $G$.
We can order the points of $\pp$ in such a way the points of $\pp^+$ proceed the points of $\pp^-$.
Then we can write $A$ as the block matrix
\[
 A = \begin{pmatrix}
  A_{++} & A_{+-} \\ A_{-+} & A_{--}
 \end{pmatrix},
\]
where $A_{\delta\eps}$ is a matrix whose rows are indexed by the points of $\pp^\delta$ and whose columns are indexed by the points of $\pp^\eps$.

Let $\one_\eps$ denote the all-one vector indexed by the points of $\pp^\eps$.
Take a point $X \in \pp^\eps$.
Then $X^\perp$ contains $q^{\frac n2 -1} \left( q^\frac n2 - \eps \right)$ anisotropic points, equally many from both quadratic types.
Therefore,
\[
 A_{\eps,\delta} \one_\delta = \frac 12 q^{\frac n2 - 1} \left( q^\frac n2 - \eps \right) \one_\eps,
\]
for all $\eps, \delta = \pm 1$.
It follows that 
$v_1 = \begin{pmatrix}
 \left( q^\frac n2 - 1 \right) \one_+ \\ \left( q^\frac n2 + 1 \right) \one_-
\end{pmatrix}$
and
$ v_2 = \begin{pmatrix} \one_+ \\ - \one_- \end{pmatrix} $
are eigenvectors of $A$ with respective eigenvalues $q^{n-1}$ and $0$.
We can extend $v_1$ and $v_2$ to an orthogonal basis of eigenvectors of $A$, which means that all other vectors of this basis are of the form $w = \begin{pmatrix} w_+ \\ w_- \end{pmatrix}$ with $w_\eps$ orthogonal to $\one_\eps$ for $\eps = \pm1$.

As before, we will proceed by examining the spectrum of $A^2$.
Since $A^2(X,Y)$ equals the number of anisotropic points in $X^\perp \cap Y^\perp$, we find that
\[
 A^2(X,Y) = \begin{cases}
  q^{\frac n2 - 1} \left( q^\frac n2 - \eps \right) & \text{if } X = Y \in \pp^\eps, \\
  q^{\frac n2 - 1} \left( q^{\frac n2 - 1} - \eps \right) & \text{if $\vspan{X,Y}$ is a tangent line, $X,Y \in \pp^\eps$}, \\
  q^{n-2} & \text{otherwise}.
 \end{cases}
\]
Recall the graphs from \Cref{SubSecGeomSchemesOdd} defined on $\pp^\eps$, where adjacency is given by lying on a tangent line.
Let $T_\eps$ denote the corresponding adjacency matrix.
Then
\begin{align*}
 A^2 = q^{n-2} J + q^{n-2}(q-1) I + q^{\frac n2 - 1} \begin{pmatrix}
  - I_+ - T_+ & O \\
  O &  I_- + T_-
 \end{pmatrix}.
\end{align*}
Now suppose that $\one_\eps, w_{\eps,2}, \dots, w_{\eps,|\pp^\eps|}$ is an orthonormal basis of eigenvectors of $T_\eps$ for $\eps = \pm 1$.
Then it readily follows that
\[
 v_1, v_2, 
 \begin{pmatrix} w_{+,2} \\ \zero \end{pmatrix}, \dots,
 \begin{pmatrix} w_{+,|\pp^+|} \\ \zero \end{pmatrix},
 \begin{pmatrix} \zero \\ w_{-,2} \\ \end{pmatrix}, \dots,
 \begin{pmatrix} \zero \\ w_{-,|\pp^-|} \\ \end{pmatrix}
\]
is an orthonormal basis of eigenvectors of $A^2$.
Note also that if $w_{\eps,i}$ has eigenvalue $\lambda$ with respect to $T_\eps$, then its corresponding eigenvector of $A^2$ has eigenvalue $q^{n-2}(q-1) - \eps q^{\frac n2 - 1}(\lambda + 1)$.
Combining this with \Cref{Tab:SpecSRGParabolic}, we see that the spectrum of $A^2$ is given by
\begin{center}
 \begin{tabular}{c || c | c | c | c}
  Eigenvalue & $q^{2(n-1)}$ & $q^{n-1}$ & $q^{n-2}$ & 0 \\ \hline
  Multiplicity & 1 & $q^n - 2 - q \theta_{n-2}(q)$ & $q \theta_{n-2}(q)$ & $1$
 \end{tabular}\:.
\end{center}
Define
\begin{align*}
 \lambda_0 = q^{n-1}, &&
 \lambda_1 = \sqrt q q^{\frac n2 - 1}, &&
 \lambda_2 = -\sqrt q q^{\frac n2 - 1}, &&
 \lambda_3 = q^{\frac n2 - 1}, &&
 \lambda_4 = -q^{\frac n2 - 1}, &&
 \lambda_5 = 0.
\end{align*}
Let $m_i$ denote the multiplicity of $\lambda_i$ as eigenvalue of $A$.
Note that $m_0 = m_5 = 1$, $m_1+m_2 = q^n-2-q\theta_{n-2}(q)$, and $m_3+m_4 = q\theta_{n-2}(q)$.
We finish the calculations again by applying \Cref{Eq:Trace} with $k=1$ and $k=3$.
To this end, we count the number of ordered triangles $(X,Y,Z)$ in $G$, which equals $\Tr(A^3)$.
There are $\frac 12 q^\frac n2 \left( q^\frac n 2 + \eps\right)$ choices for $X \in \pp^\eps$.
There are $q^{\frac n2 - 1} \left( q^\frac n2 - \eps \right)$ anisotropic points $Y \in X^\perp$.
Note that the line $\vspan{X,Y}$ is not tangent, since it doesn't meet $X^\perp$ in a point of $\mq$.
Hence, $\vspan{X,Y}^\perp \cap \mq \cong \mq(n-2,q)$, which means there are $q^{n-2}$ choices for $Z$.
We conclude that
\[
 \sum_i m_i \lambda_i^3 = \Tr(A^3)
 = \sum_{\eps = \pm 1} \frac 12 q^\frac n2 \left( q^\frac n2 + \eps \right) q^{\frac n2 - 1} \left( q^\frac n2 - \eps \right) q^{n-2}
 = q^{2n-3} (q^n-1).
\]
The system of linear equations in $m_0,\dots,m_5$ has as unique solution
\begin{align*}
 m_0 &= m_5 = 1, &
 m_1 &= m_2 = \frac 1 2 \left( q^n - q \theta_{n-2}(q) \right) - 1, \\
 m_3 &= \frac 12 \left( q \theta_{n-2}(q) - q^\frac n2 \right), &
 m_4 &= \frac 12 \left( q \theta_{n-2}(q) + q^\frac n2 \right).
\end{align*}
This proves \Cref{Thm:EigOrtho} (\ref{Thm:EigOrtho:Par}).

\paragraph{Acknowledgements.}
The first author is very grateful for the hospitality of the Faculty of Mathematics at the University of Rijeka during his visit to the second author, during which the research of the present article was initiated. The authors are also grateful for helpful discussions with Jan De Beule and Ferdinand Ihringer, and to Edwin van Dam for pointing out \cite{FisherPenttilaPraegerRoyle} to us.
This work has been partially supported by Croatian Science Foundation under the project 5713.


\begin{thebibliography}{BSHW91}

\bibitem[Adr22]{adriaensen2022stability}
S.~Adriaensen.
\newblock Stability of {E}rd{\H o}s-{K}o-{R}ado theorems in circle geometries.
\newblock {\em J. Combin. Des.}, 30(11):689--715, 2022.

\bibitem[AK97]{AlonKrivelevich}
N.~Alon and M.~Krivelevich.
\newblock Constructive bounds for a {R}amsey-type problem.
\newblock {\em Graphs Combin.}, 13(3):217--225, 1997.

\bibitem[Ban91]{bannai1991}
E.~Bannai.
\newblock Subschemes of some association schemes.
\newblock {\em J. Algebra}, 144(1):167--188, 1991.

\bibitem[BC95]{BuekenhoutCameron}
F.~Buekenhout and P.~Cameron.
\newblock Chapter 2 - {P}rojective and affine geometry over division rings.
\newblock In F.~Buekenhout, editor, {\em Handbook of Incidence Geometry}, pages
  193--294. North-Holland, Amsterdam, 1995.

\bibitem[BCN89]{bcn}
A.~E. Brouwer, A.~M. Cohen, and A.~Neumaier.
\newblock {\em Distance-regular graphs}, volume~18 of {\em Ergebnisse der
  Mathematik und ihrer Grenzgebiete (3) [Results in Mathematics and Related
  Areas (3)]}.
\newblock Springer-Verlag, Berlin, 1989.

\bibitem[BHS90]{bannaihaosong}
E.~Bannai, S.~Hao, and S.-Y. Song.
\newblock Character tables of the association schemes of finite orthogonal
  groups acting on the nonisotropic points.
\newblock {\em J. Combin. Theory Ser. A}, 54(2):164--200, 1990.

\bibitem[BIP20]{BishnoiIhringerPepe}
A.~Bishnoi, F.~Ihringer, and V.~Pepe.
\newblock A construction for clique-free pseudorandom graphs.
\newblock {\em Combinatorica}, 40(3):307--314, 2020.

\bibitem[BS74]{buekenhoutshult}
F.~Buekenhout and E.~Shult.
\newblock On the foundations of polar geometry.
\newblock {\em Geometriae Dedicata}, 3:155--170, 1974.

\bibitem[BSHW91]{bannaisonghaowei}
E.~Bannai, S.-Y. Song, S.~Hao, and H.~Z. Wei.
\newblock Character tables of certain association schemes coming from finite
  unitary and symplectic groups.
\newblock {\em J. Algebra}, 144(1):189--213, 1991.

\bibitem[BvL84]{brouwervanlint}
A.~E. Brouwer and J.~H. van Lint.
\newblock Strongly regular graphs and partial geometries.
\newblock In {\em Enumeration and design ({W}aterloo, {O}nt., 1982)}, pages
  85--122. Academic Press, Toronto, ON, 1984.

\bibitem[BVM22]{brouwervanmaldeghem}
A.~E. Brouwer and H.~Van~Maldeghem.
\newblock {\em Strongly regular graphs}, volume 182 of {\em Encyclopedia of
  Mathematics and its Applications}.
\newblock Cambridge University Press, Cambridge, 2022.

\bibitem[CDBS13]{CoolsaetDeBeuleSiciliano}
K.~Coolsaet, J.~De~Beule, and A.~Siciliano.
\newblock The known maximal partial ovoids of size {$q^2-1$} of {$Q(4,q)$}.
\newblock {\em J. Combin. Des.}, 21(3):89--100, 2013.

\bibitem[Cha71]{chakravarti71}
I.~M. Chakravarti.
\newblock Some properties and applications of {H}ermitian varieties in a finite
  projective space {$PG(N,\,q^{2})$} in the construction of strongly regular
  graphs (two-class association schemes) and block designs.
\newblock {\em J. Combinatorial Theory Ser. B}, 11:268--283, 1971.

\bibitem[DBG08]{DeBeuleGacs}
J.~De~Beule and A.~G\'{a}cs.
\newblock Complete arcs on the parabolic quadric {${\rm Q}(4,q)$}.
\newblock {\em Finite Fields Appl.}, 14(1):14--21, 2008.

\bibitem[Del95]{delandtsheer}
A.~Delandtsheer.
\newblock Chapter 6 - {D}imensional linear spaces.
\newblock In F.~Buekenhout, editor, {\em Handbook of Incidence Geometry}, pages
  193--294. North-Holland, Amsterdam, 1995.

\bibitem[FPPR89]{FisherPenttilaPraegerRoyle}
P.~H. Fisher, T.~Penttila, C.~E. Praeger, and G.~F. Royle.
\newblock Inversive planes of odd order.
\newblock {\em European J. Combin.}, 10(4):331--336, 1989.

\bibitem[GL81]{gordonlevingston}
L.~M. Gordon and R.~Levingston.
\newblock The construction of some automorphic graphs.
\newblock {\em Geom. Dedicata}, 10(1-4):261--267, 1981.

\bibitem[GM16]{godsilmeagher}
C.~Godsil and K.~Meagher.
\newblock {\em Erd\H{o}s-{K}o-{R}ado theorems: algebraic approaches}, volume
  149 of {\em Cambridge Studies in Advanced Mathematics}.
\newblock Cambridge University Press, Cambridge, 2016.

\bibitem[Har04]{hartmann}
E.~Hartmann.
\newblock Planar circle geometries.
\newblock
  \url{https://www2.mathematik.tu-darmstadt.de/~ehartmann/circlegeom.pdf},
  2004.

\bibitem[HT16]{hirschfeldthas}
J.~W.~P. Hirschfeld and J.~A. Thas.
\newblock {\em General {G}alois geometries}.
\newblock Springer Monographs in Mathematics. Springer, London, 2016.

\bibitem[Shu11]{Shult11}
E.~E. Shult.
\newblock {\em Points and lines: characterizing the classical geometries}.
\newblock Universitext. Springer, Heidelberg, 2011.

\bibitem[Van20]{vanhove20}
F.~Vanhove.
\newblock The association scheme on the points off a quadric.
\newblock {\em Bull. Belg. Math. Soc. Simon Stevin}, 27(1):153--160, 2020.
\newblock With a conclusion by A. E. Brouwer.

\bibitem[Wei83]{wei83}
H.~Z. Wei.
\newblock Using finite unitary geometry to construct a class of {PBIB} designs.
\newblock {\em Chinese Ann. Math. Ser. B}, 4(3):299--306, 1983.
\newblock A Chinese summary appears in Chinese Ann. Math. Ser. A {{\bf{4}}}
  (1983), no. 4, 544.

\end{thebibliography}
\end{document}